\documentclass[headings]{article}

\usepackage[utf8]{inputenc}
\usepackage[french]{babel}

\usepackage{amsthm}

\usepackage{fullpage}

\usepackage{fancyhdr}

\usepackage{makeidx}

\usepackage{amsmath}
\usepackage{amssymb}
\usepackage{frcursive}
\usepackage{mathrsfs}
\usepackage{latexsym}
\usepackage{stmaryrd}
\usepackage{dsfont} % Fonctions indicatrices
\usepackage[all]{xy} % Diagrammes

\usepackage{boxedminipage}

\usepackage{listings}

\usepackage{minitoc}

\usepackage{ifpdf}

\ifpdf
\usepackage[pdftex]{graphicx}
\else
\usepackage{graphicx}
\fi

\title{Un théorème de Paley-Wiener matriciel\\ pour un groupe réductif $p$-adique non connexe}
\author{Joël Cohen\,\footnote{Ce travail a bénéficié d'une bourse ANR blanc JIVARO (référence ANR-08-BLAN-0259-02).}}

\date{\today}

\newtheorem{thm}{Théorème}[section]
\newtheorem{defn}[thm]{Définition}

\newtheorem{prop}[thm]{Proposition}
\newtheorem{lemme}[thm]{Lemme}
\newtheorem{rem}[thm]{Remarque}
\newtheorem{ex}[thm]{Exemple}

\newcommand{\Z}{\mathbb{Z}}

\newcommand{\R}{\mathbb{R}}
\newcommand{\C}{\mathbb{C}}

\newcommand{\Orb}{\mathcal{O}}

\newcommand{\Hecke}{\mathcal{H}}
\newcommand{\hecke}{\mathcal{C}_{c}^{\infty}}
\newcommand{\cat}{\mathcal{C}}
\newcommand{\Ind}{\operatorname{Ind}}

\newcommand{\Irr}{\operatorname{Irr}}
\newcommand{\GL}{\operatorname{GL}}

\newcommand{\Hom}{\operatorname{Hom}}
\newcommand{\End}{\operatorname{End}}

\newcommand{\Supp}{\operatorname{Supp}}

\newcommand{\Rat}{\operatorname{Rat}}
\newcommand{\coker}{\operatorname{coker}}
\newcommand{\Rep}{\mathcal{R}} % Représentation
\newcommand{\efo}{\mathcal{A}} % Endomorphismes du foncteur d'oubli
\newcommand{\D}{\operatorname{d}} % d (droit)
\newcommand{\pars}[1]{\ensuremath{\left( #1 \right)}} % (*)
\newcommand{\braces}[1]{\ensuremath{\left\{ #1 \right\}}} % {*}
\newcommand{\scal}[1]{\ensuremath{\left\langle #1 \right\rangle}} % <*>
\newcommand{\ints}[1]{\ensuremath{\left[\!\left[ #1 \right]\!\right]}} % [|*|]
\def\revdots{\mathinner{\mkern1mu\raise1pt\vbox{\kern7pt\hbox{.}}\mkern2mu\raise4pt\hbox{.}\mkern2mu \raise7pt\hbox{.}\mkern1mu}} % Trois petit point en diagonale montante

\newcommand{\Tr}{\operatorname{Tr}}
\newcommand{\Id}{\operatorname{Id}}
\newcommand{\ad}{\operatorname{ad}}
\newcommand{\Ad}{\operatorname{Ad}}

\newcommand{\et}{\qquad \textnormal{ et } \qquad}

\begin{document}

\ifpdf
\DeclareGraphicsExtensions{.pdf, .jpg, .tif}
\else
\DeclareGraphicsExtensions{.eps, .jpg}
\fi

\maketitle

Soit $F$ un corps local non archimédien de caractéristique nulle, et $G$ (le groupe des points sur $F$ d') un groupe réductif défini sur $F$ non nécessairement connexe. On note $G^0$ sa composante connexe de l'identité, on suppose le quotient $G/G^0$ commutatif. Pour $f : G \to \C$ une fonction lisse à support compact et $(\pi, V_{\pi}) \in \Rep(G)$ une représentation lisse de $G$, on fixe une mesure de Haar $\D g$ sur $G$, et on définit la transformée de Fourier de $f$ en $\pi$, par

\[
	\widehat{f}(\pi) = \int_{G} \pi(g) \, f(g) \, \D g
\]
C'est un endomorphisme de l'espace vectoriel $V_{\pi}$ sous-jacent à $\pi$. On fournit une description de l'image de la transformée de Fourier, c'est-à-dire étant donné, pour toute représentation $(\pi, V_{\pi})$ dans une certaine famille d'induites, un endomorphisme $\varphi(\pi)$ de l'espace vectoriel $V_{\pi}$, à quel condition nécessaire et suffisante existe-t-il une fonction $f$ lisse à support compact telle que $\varphi(\pi) = \widehat{f}(\pi)$ pour tout $\pi$ dans la famille ?

Ce travail fait partie de ma thèse de doctorat préparée sous la direction de Volker Heiermann. Je remercie Jean-Loup Waldspurger pour les corrections et commentaires utiles qu'il m'a apporté.

\tableofcontents

\ \\
\ \\

\section{Introduction} % (fold)
\label{sec:introduction}

Soit $F$ un corps local non archimédien localement compact (i.e. de corps résiduel fini) de caractéristique $0$\,\footnote{Cette hypothèse de caractéristique nulle n'est probablement pas nécessaire, nous l'avons incluse par prudence car nous utilisons les résultats de certains articles qui le supposent.}. Dans tout le document on commet l'abus de langage courant qui consiste à identifier les groupes algébriques et les groupes de leurs points sur $F$. Soit $G$ (le groupe des points sur $F$ d') un groupe réductif défini sur $F$. On note $G^0$ sa composante connexe de l'identité, dite composante neutre. On sait que c'est un sous-groupe distingué d'indice fini. On supposera de plus que le quotient $G/G^0$, dit groupe des composantes, est commutatif. On fixe $S \subset G$ un système de représentants de $G/G^0$. 

On note $\Rep(G)$ la catégorie des représentations complexes lisses de $G$, et $\Hecke(G) = \hecke(G)$ l'algèbre de Hecke des distributions complexes localement constantes à support compact munies du produit de convolution par rapport à $\D g$ (remarquons que $G$ est unimodulaire parce que $G^0$ l'est). On adoptera sans précision des notations analogues pour n'importe quel groupe localement compact totalement discontinus (en particulier pour les sous-groupes fermés de $G$).

Le groupe $G$ agit sur lui-même par adjonction. C'est-à-dire que si $g \in G$, on note $g \cdot x = \Ad(g)(x) = gxg^{-1}$ et $\Ad$ est une action continue. Il s'en déduit une action sur les parties de $G$, que l'on note encore par un point $g \cdot X = gXg^{-1}$ qui induit une action sur les sous-groupes de $G$. Si $H \subset G$ est un sous-groupe fermé de $G$, et $\sigma \in \Rep(H)$ une représentation lisse de $H$, alors on définit une représentation lisse de $g \cdot H$, notée $g \cdot \sigma$ par $(g \cdot \sigma)(g \cdot h) = \sigma(h)$. En particulier, cela définit une action sur les représentations de chaque sous-groupe distingué de $G$ (dont $G$ lui-même).

Pour $f \in \Hecke(G)$, on définit la transformée de Fourier de $f$ en $(\pi, V_{\pi}) \in \Rep(G)$, par

\[
	\widehat{f}(\pi) = \int_{G} \pi(g) \, f(g) \, \D g
\]
C'est un endomorphisme de l'espace vectoriel $V_{\pi}$ sous-jacent à $\pi$. Notre but est de caractériser l'image de la transformée de Fourier : étant donné une sous-catégorie $\cat$ pleine de $\Rep(G)$ (le domaine de définition) suffisamment grande, et la donnée pour toute représentation $\pi \in \cat$ d'un endomorphisme $\varphi(\pi) \in \End_{\C}(\pi)$, à quelles conditions $\varphi$ est-il de la forme $\varphi = \widehat{f}$ avec $f \in \Hecke(G)$.

Dans le cas où $G$ est connexe, c'est un théorème connu (voir \cite{BDK} pour le théorème portant sur la trace de cet opérateur dans le cas connexe, \cite{Rogawski} pour une version non connexe, et \cite{Heiermann} et \cite{Bernstein} théorème 25 pour la version matricielle dans le cas connexe), la transformée de Fourier vérifie 3 conditions : lissité (ou finitude), commutation aux entrelacements, polynomialité. Dans le cas non connexe (sous l'hypothèse que $G/G^0$ est commutatif), on fournit une description analogue.

\subsection{Le cas connexe} % (fold)
\label{sub:le_cas_connexe}

On suppose dans cette section que $G$ est connexe. On considère la famille des endomorphismes $\varphi_{P, \sigma} = \widehat{f}(\Ind_{P}^{G} \sigma)$ où $P = MN$ est un sous-groupe parabolique de $G$ de facteur de Levi $M$, et $\sigma \in \Rep(M)$ une (classe d'isomorphisme de) représentation irréductible cuspidale de $M$. On note $\mathcal{B}(G)$ l'ensemble de tels couples $(P, \sigma)$ et $\cat$ la sous-catégorie pleine de $\Rep(G)$ stable par somme et sous-quotient engendrée (au sens de la définition \ref{defn:engendrer}) par les représentations $\Ind_{P}^{G} \sigma$. La transformée de Fourier $\varphi = \widehat{f}$ vérifie trois conditions : lissité (ou finitude), commutation aux entrelacements, polynomialité.

\subsubsection{Lissité} % (fold)
\label{ssub:lissite}
Il existe un sous-groupe ouvert compact $K \subset G$ par lequel $f$ est bi-invariant. En prenant la transformée de Fourier, on trouve

\[
	\varphi = \varphi \, \widehat{e_K} = \widehat{e_K} \, \varphi
\]
Où $e_K \in \Hecke(G)$ désigne la mesure de Haar normalisée sur $K$ (prolongée à $G$ par $0$). Ainsi $\varphi(\pi)$ est donc entièrement déterminé par l'endomorphisme sur l'espace des vecteurs $K$-fixes par restriction (en particulier $\varphi$ est de rang fini sur les représentations admissibles). On dira que $\varphi = \widehat{f}$ est \textbf{lisse}.
% subsubsection lissité (end)

\subsubsection{Commutation aux entrelacements} % (fold)
\label{ssub:commutation_aux_entrelacements}

Si $\pi_1, \pi_2 \in \cat$ sont des représentations, et $\alpha \in \Hom_{G}(\pi_1, \pi_2)$ un $G$-entrelacement entre elles, alors
\[
	\alpha \circ \varphi(\pi_1) = \varphi(\pi_2) \circ \alpha
\]
On dira que $\varphi$ \textbf{commute aux $G$-entrelacements}. C'est une condition qui assure que la donnée de $\varphi(\pi)$ est naturelle donc ne dépend pas du choix d'une réalisation de $\pi$ (formellement, $\varphi$ définit un endomorphisme du foncteur d'oubli $\cat \to \C$-Vect). La propriété permet d'étendre naturellement $\varphi$ aux sous-quotients et sommes directes (ce qui permet d'agrandir $\cat$ pour la rendre stable par ces opérations) de telle sorte que le prolongement continue à commuter aux entrelacements, et coïncide avec $\widehat{f}$ si c'était le cas pour $\varphi$.
% subsubsection commutation_aux_entrelacements (end)

\subsubsection{Polynomialité} % (fold)
\label{subs:polynomialite}

Si $M$ est un sous-groupe de Levi de $G$, le groupe $X(M)$ des caractères non ramifiés (c'est-à-dire triviaux sur ses sous-groupes compacts) de $M$ est un tore complexe. On vérifie alors que la fonction $\chi \mapsto \widehat{f}(\Ind_{P}^{G} \chi\sigma)$ est polynomiale\,\footnote{Pour parler précisément de polynomialité, on réalise toutes les représentations $\Ind_{P}^{G} \chi\sigma$ dans le même espace vectoriel : Si $K_0$ est un bon sous-groupe compact maximal de $G$ (c'est-à-dire tel que $G = P K_0$), alors la restriction à $K_0$ est un $K_0$-isomorphisme entre $\Ind_{P}^{G} \chi \sigma$ et $\Ind_{P \cap K_0}^{K_0} \sigma$ (voir \cite{Heiermann} p.1).}.

% subsubsection polynomialite (end)

\subsubsection{Le théorème dans le cas connexe} % (fold)
\label{ssub:le_theoreme}

Réciproquement on la théorème suivant (voir \cite{Heiermann} théorème 0.1 pour plus de détails).

\begin{thm}[voir \cite{Heiermann} et \cite{Bernstein} théorème 25]
	\label{thm:paley_wiener_connexe}
	 Etant donné, pour tout couple $(P, \sigma) \in \mathcal{B}(M)$, un endomorphisme $\varphi_{P, \sigma} \in \End_{\C}(\Ind_{P}^{G} \sigma)$ vérifiant les conditions suivantes :
	\begin{enumerate}
		\item \textbf{Lissité (ou finitude) :} Il existe un sous-groupe ouvert compact $K \subset G$ tel que $\varphi$ est bi-invariant par $K$ :
		\[
			\varphi = \varphi \, \widehat{e_K} = \widehat{e_K} \, \varphi
		\]
		\item \textbf{Polynomialité :} L'application $\chi \longmapsto \varphi_{P, \chi\sigma}$ est polynomiale.
		\item \textbf{Commutation aux entrelacements :}
			\begin{enumerate}
				\item \textbf{Translation à gauche : } Pour tout $g \in G$,
%%%% Qui est W(M)				
				
				\begin{equation}
					\lambda(g) \circ \varphi_{P, \sigma} = \varphi_{g \cdot P, g \cdot \sigma} \circ \lambda(g)
				\end{equation}
				où $\lambda(g)$ désigne la translation à gauche des fonctions.
				\item \textbf{Opérateurs d'entrelacement} Si $P' = M N'$ est un autre sous-groupe parabolique de composante de Levi $M$,
				\begin{equation}
					J_{P|(P')}(\sigma) \circ \varphi_{P', \sigma} = \varphi_{P, \sigma} \circ J_{P|P'}(\sigma)
				\end{equation}
				Où l'entrelacement $J_{P|P'}(\sigma)$ est défini, sous réserve de convergence\footnote{C'est-à-dire que si l'on fixe $\sigma$, l'opérateur $J_{P|P'}(\chi\sigma)$ est défini par l'intégrale ci-dessous pour $\chi \in X(M)$ dans un certain cône de convergence.}, par
				\[
					J_{P|P'}(\sigma) = \int_{N / N \cap N'} \lambda(n) \, \D n 
				\]
			\end{enumerate}
	\end{enumerate}
	Alors il existe une unique fonction $f \in \Hecke(G)$ telle que pour tout $(P, \sigma) \in \mathcal{B}(M)$, on ait
	\[
		\varphi_{P, \sigma} = \widehat{f}(\Ind_{P}^{G} \sigma)
	\]
	Par ailleurs, on dispose d'une formule d'inversion (voir \cite{Heiermann} proposition 0.2).
\end{thm}

% subsubsection le_théorème (end)

% subsection le_cas_connexe (end)
\subsection{Le cas non connexe} % (fold)
\label{sub:un_resultat_dans_le_cas_non_connexe}

Expliquons désormais le thèorème de Paley-Wiener dans le cas non connexe. Remarquons d'abord que les conditions de lissité et commutation aux entrelacements sont conservées, car valables sur tout groupe localement profini, de même que l'injectivité de la transformée de Fourier (c'est le lemme de séparation, voir \cite{Zelevinskii} proposition 2.12% \cite{Renard} III.1.11
). Il reste à trouver un domaine de définition convenable pour la transformée de Fourier, c'est-à-dire une bonne famille de représentations\,\footnote{On demande notamment que les irréductibles soient des sous-représentations d'éléments de cette famille, que l'on sache décrire leurs entrelacements afin de formuler la condition d'entrelacement de manière explicite, et que l'on possède une structure de variété afin d'avoir un analogue de la condition de polynomialité.} de $G$. On adopte pour cela la notion de sous-groupe paraboliques dits \emph{cuspidaux} développée dans \cite{Goldberg}, et la théorie de l'induction parabolique qui en découle. Géométriquement, ils sont construits à partir des sous-groupes paraboliques de $G^0$ : Soit $P^0 = M^0 N$ un sous-groupe parabolique de $G^0$ de composante de Levi $M^0$ et $A$ le tore maximal déployé dans le centre de $M^0$. Alors $M = C_G(A)$, le centralisateur de $A$ dans $G$, est un sous-groupe de Levi cuspidal de $G$, et $P = M N$ est un sous-groupe parabolique cuspidal de $G$. Ces groupes permettent d'avoir une notion de représentations cuspidales et de prolonger des résultats usuels sur l'induction parabolique à la catégorie $\Rep(G)$.  On définit notamment un groupe $X^0(M)$ de caractères non ramifiés qui est un tore complexe (donc une notion de régularité vis-a-vis des paramètres complexes de ces caractères), et des opérateurs d'entrelacements $J$ analogues à ceux construits dans le cas connexe (voir la partie \ref{sec:induction_parabolique_dans_le_groupe_g} pour les définitions). On note $W(M) = N_G(M)/M$. On notera à nouveau $\mathcal{B}(G)$ l'ensemble des couples $(P, \sigma)$ avec $P = MN$ un sous-groupe parabolique cuspidal de $G$ et $\sigma \in \Rep(M)$ une (classe d'équivalence de) représentation irréductible cuspidale de $M$ (voir définition en partie \ref{sec:induction_parabolique_dans_le_groupe_g}), et $\cat$ la sous-catégorie pleine de $\Rep(G)$ stable par somme et sous-quotient engendrée (au sens de la définition \ref{defn:engendrer}) par les représentations $\pi = \Ind_P^G(\sigma)$. On obtient alors un résultat analogue au cas connexe.

% \begin{thm}
% 	\label{thm:paley_wiener}
% 	Avec ces nouvelles définitions (pour l'ensemble $\mathcal{B}(G)$, pour la notion de polynomialité, pour les opérateurs d'entrelacement $J_{P,P'}(\sigma)$), l'énoncé du théorème \ref{thm:paley_wiener_connexe} est valide dans le cas non connexe.
% \end{thm}
% Description
\begin{thm}
	\label{thm:paley_wiener}
	 Etant donné, pour tout couple $(P, \sigma) \in \mathcal{B}(M)$, un endomorphisme $\varphi_{P, \sigma} \in \End_G(\Ind_P^G \sigma)$ vérifiant les conditions suivantes :
	\begin{enumerate}
		\item \textbf{Lissité (ou finitude) :} Il existe un sous-groupe ouvert compact $K \subset G$ tel que $\varphi$ est bi-invariant par $K$ :
		\[
			\varphi = \varphi \, \widehat{e_K} = \widehat{e_K} \, \varphi
		\]
		\item \textbf{Polynomialité :} L'application $\chi \longmapsto \varphi_{P, \chi\sigma}$ définie sur $X^0(M)$ est polynomiale.
		\item \textbf{Commutation aux entrelacements :}
			\begin{enumerate}
				\item \textbf{Translation à gauche : } Si $g \in G$ est tel que $\ad_g$ envoie $P,M, \sigma$ sur $P', M', \sigma'$, alors
				\begin{equation}
					\lambda(g) \circ \varphi_{P, \sigma} = \varphi_{P', \sigma'} \circ \lambda(g)
				\end{equation}
				où $\lambda(g)$ désigne la translation à gauche des fonctions.
				\item \textbf{Opérateurs d'entrelacement} Si $P' = M N'$ est un autre sous-groupe parabolique de composante de Levi $M$,
				\begin{equation}
					J_{P|P'}(\sigma) \circ \varphi_{P', \sigma} = \varphi_{P, \sigma} \circ J_{P|P'}(\sigma)
				\end{equation}
				Où l'entrelacement $J_{P|P'}(\sigma)$ est défini, sous réserve de convergence, par
				\[
					J_{P|P'}(\sigma) = \int_{N / N \cap N'} \lambda(n) \, \D n 
				\]
			\end{enumerate}
	\end{enumerate}
	Alors il existe une unique fonction $f \in \Hecke(G)$ telle que pour tout $(P, \sigma) \in \mathcal{B}(M)$, on ait
	\[
		\varphi_{P, \sigma} = \widehat{f}(\Ind_P^G \sigma)
	\]
	En outre on dispose d'une formule d'inversion pour $f$ que l'on explicite en \ref{ssub:une_formule_pour_f_}.
\end{thm}

Notre méthode consiste à déduire ce résultat du cas connexe. Décrivons succinctement le plan. Dans la partie \ref{sec:endomorphismes_du_foncteur_d_oubli}, on établit une méthode générique permettant de relier les propriétés de la transformée de Fourier d'un groupe localement profini $G$ à celles d'un de ses sous groupes d'indice fini. Dans les parties \ref{sec:induction_parabolique_dans_le_groupe_g} et \ref{sec:representations_d_une_groupe_reductif_non_connexe}, on revient dans le cadre d'un groupe réductif sur $F$, et on rappelle les définitions et les résultats principaux de \cite{Goldberg} concernant l'induction parabolique dans ce cadre tordu. Dans la partie \ref{sec:le_theoreme}, on assemble ces éléments pour donner la preuve du théorème \ref{thm:paley_wiener}.

% subsection un_résultat_dans_le_cas_connexe (end)

\section{Endomorphismes du foncteur d'oubli} % (fold)
\label{sec:endomorphismes_du_foncteur_d_oubli}

Pour cette partie on se place dans un cadre un peu plus général que dans le reste du document. Soit $G$ un groupe topologique localement profini. Soit $H \subset G$ un sous-groupe distingué de $G$ d'indice fini ouvert et fermé\,\footnote{Ces hypothèses sont en fait redondantes puisque le fait d'être ouvert entraine d'être fermé, et le fait d'être fermé d'indice fini entraîne d'être ouvert.}. Les résultats de cette partie seront utilisés pour $G$ un groupe réductif $p$-adique, et $H$ sa composante neutre.

Comme dans les autres parties, on note $\Rep(G)$ la catégorie des représentations complexes lisses de $G$. Si $\pi \in \Rep(G)$ on note $\text{JH}(\pi)$ l'ensemble des classes d'isomorphisme de sous-quotients irréductibles de $\pi$. On se donne $\cat$ une sous-catégorie pleine de $\Rep(G)$ (pour l'instant arbitraire, mais l'on ajoutera vite des conditions supplémentaires). A terme, nous aurons en tête pour $\cat$ la catégorie engendrée (au sens précisé à la définition \ref{defn:engendrer}) par les représentations du type $\Ind_P^G(\sigma)$ pour $P = MN$ parcourant l'ensemble des sous-groupes paraboliques de $G$, et $\sigma \in \Rep(M)$ celui des représentations cuspidales irréductibles de $M$.

On dispose d'un foncteur d'oubli
\begin{align*}
	r_{1,\cat}^G : \quad \cat \quad &\longrightarrow \C- \textrm{Vect}\\
	(\pi, V) &\longmapsto V
\end{align*}
Et on note $\efo(\cat) = \End(r^G_{1,\cat})$ l'algèbre des endomorphismes de ce foncteur. Un endomorphisme $\varphi \in \efo(\cat)$ est la donnée pour toute représentation $(\pi, V) \in \cat$ d'un endomorphisme $\varphi(\pi) \in \End_{\C}(\pi)$ de l'espace vectoriel sous-jacent $V$, tel que pour tout $\pi_1, \pi_2 \in \cat$ et tout $G$-morphisme $\alpha \in \Hom_G(\pi_1, \pi_2)$, on ait $\alpha \circ \varphi(\pi_1) = \varphi(\pi_2) \circ \alpha$. On appellera cette condition la commutation aux $G$-entrelacements.

L'ensemble $\efo(G)$ forme une $\C$-algèbre unitaire, où les opérations se font ``terme à terme'' (le produit étant la composition des endomorphismes). En particulier l'élément neutre est $(\pi, V) \mapsto \Id_V$ %, et $\efo(G)^{\times}$ est l'ensemble des $\varphi \in \efo(G)$ tels que $\varphi(\pi)$ est inversible pour tout $(\pi, V) \in \Rep(G)$.
et pour $g \in G$, on pose

\[
	\widehat{\delta_g} : (\pi,V) \in \cat \mapsto \pi(g)
\]
Une seconde famille d'éléments de $\efo(G)$ est donnée par la transformée de Fourier des fonctions lisses à support compact, comme on l'explique dans la section suivante.

\subsection{Transformée de Fourier} % (fold)
\label{sub:transformee_de_fourier}

On choisit dans toute la suite une mesure de Haar sur $G$. Si $F \in \Hecke(G)$, on définit un élément $\widehat{F}_{\cat} \in \efo(\cat)$, appelée transformée de Fourier de $F$, pour $\pi \in \cat$ par

\[
	\widehat{F}_{\cat}(\pi) = \int_G \pi(g) F(g) \, \D g
\]
La transformée de Fourier $F \mapsto \widehat{F}_{\cat}$ est un morphisme d'algèbres de $\Hecke(G)$ dans $\efo(\cat)$. Pour alléger la notation, on fera disparaître l'indice $\cat$, et on notera simplement $\widehat{F} \in \efo(\cat)$. De manière plus générale, on peut donner un sens à la transformée de Fourier d'une famille plus large de distributions, quitte éventuellement à changer la catégorie $\cat$ des représentations considérées (pour pouvoir donner un sens à l'intégrale ci-dessus). On peut notamment donner le sens que nous avons donné plus haut à la transformée de Fourier des distributions de Dirac $\delta_g$ en $g \in G$. Si $\cat$ est la catégorie de représentations tempérées, alors on peut donner un sens à cette intégrale pour toute fonction $F$ de Schwartz-Harish-Chandra : si $\xi_{v, \tilde{v}} : g \mapsto \scal{\pi(g)v, \tilde{v}}$ est un coefficient matriciel de $(\pi, V)$, alors l'intégrale sur $G$ de $F(g) \xi_{v, \tilde{v}}(g)$ converge et on définit $\pi(F)v \in V$ comme l'unique vecteur qui vérifie $\int_G F(g) \xi_{v, \tilde{v}}(d) \, \D g = \scal{\pi(F)v, \tilde{v}}$ pour tout vecteur $\tilde{v} \in \tilde{V}$ dans la représentation contragrédiente de $(\pi, V)$ (voir \cite{Waldspurger} III.7). On peut de manière générale imaginer $\efo(\cat)$ comme une algèbre dont les éléments sont des distributions formelles (et dans certains cas, comme par exemple $\cat = \Rep(G)$, on peut montrer que les éléments de $\efo(\cat)$ correspondent effectivement à des vraies distributions sur $G$, voire des fonctions cf \cite{HarishVanDijk}), plus $\cat$ est petite, plus il y en a.

\begin{lemme}
	On suppose que la catégorie $\cat$ contient toutes les représentations irréductibles. Alors la transformation de Fourier est injective.
\end{lemme}
\begin{proof}
	C'est une conséquence du lemme de séparation (voir par exemple \cite{Zelevinskii} proposition 2.12).
\end{proof}
La catégorie $\cat$ que nous utiliserons dans les applications (engendrée par les induite paraboliques de représentations cuspidales) vérifie cette propriété (puisque tout représentation irréductible est sous-représentation d'une telle induite). Notre but sera de décrire l'image de $\Hecke(G)$ dans $\efo(\cat)$, et de donner un procédé d'inversion.
% subsection transformée_de_fourier (end)

\subsection{Lissité} % (fold)
\label{sub:lissite}

Si $K \subset G$ est un sous-groupe ouvert compact, on note $e_K \in \Hecke(G)$ la fonction à support dans $K$ constante sur $K$ et d'intégrale $1$. On dit que $\varphi \in \efo(\cat)$ est $K$-bi-invariant si $\varphi = \varphi \widehat{e_K} = \widehat{e_K} \varphi$ ($\varphi$ revient alors à la donnée pour tout $(\pi,V) \in \cat$ d'un endomorphisme de l'espace $V^K$ des vecteurs $K$-fixes). On note $\efo^K(\cat)$ la sous-algèbre des éléments $K$-bi-invariants. On dira que $\varphi \in \efo(\cat)$ est lisse s'il est bi-invariant par un sous-groupe ouvert compact, et on note $\efo^{\infty}(\cat) = \bigcup_{K} \efo^K(\cat)$ la sous-algèbre des éléments lisses. La transformée de Fourier envoie $e_K \Hecke(G) e_K$ dans $\efo^K(\cat)$ et $\Hecke(G)$ dans $\efo^{\infty}(\cat)$. Par contre, si $g \in G$ n'est pas dans le centre, alors $\widehat{\delta_g}$ n'est pas lisse (sauf si $G$ est discret). Signalons au passage un résultat de densité.

\begin{lemme}
	\label{lemme:densite}
	On a
	\[
		\efo(\cat) = \varprojlim_{K} \efo^K(\cat)
	\]
	et $\efo^{\infty}(\cat)$ est dense dans $\efo(\cat)$ pour la topologie de la limite projective. Par ailleurs, $\efo^{\infty}(\Rep(G)) \simeq \Hecke(G)$.
\end{lemme}
\begin{proof}
	Voir \cite{Deligne}.
\end{proof}

% subsection lissité (end)

\subsection{Décomposition en classes modulo $H$} % (fold)
\label{sub:decomposition_en_classes_modulo_h_}

On se propose d'expliquer comment inverser la transformée de Fourier sur $G$ si on sait le faire sur $H$ (la formule est en quelque sorte une une combinaison d'inversions de Fourier sur $H$ et sur $G/H$). Fixons $S \subset G$ un système de représentants de $G/H$, alors on a la décomposition suivante (en sous-espaces vectoriels) de l'algèbre de Hecke de $G$

\[
	\Hecke(G) = \bigoplus_{s \in S} \delta_{s} * \Hecke(H)
\]
Où $\delta_s$ désigne la distribution de Dirac en $s$. Explicitement, pour $F \in \Hecke(G)$ et $s \in S$, posons $F_s = (\delta_{s^{-1}} * F) . \mathds{1}_{H}$ et $f_s = (F_s)_{|H}$ (c'est-à-dire que $F_s(h) = f_s(h) = F(sh)$ pour $h \in H$ et $F_s$ est nul hors de $H$), alors l'application $F \mapsto (f_s)_{s \in S}$ est un isomorphisme d'espaces vectoriels.
De la décomposition $F = \sum_{s \in S} \delta_{s} * F_s$, on tire via la transformée de Fourier
\[
	\widehat{F} = \sum_{s \in S} \widehat{\delta}_{s} . \widehat{F}_s
\]
où le produit est la composition terme à terme des endomorphismes. On va montrer qu'étant donné $\widehat{F}$ la transformée de Fourier, on peut construire $\widehat{F_s}$ puis $\widehat{f_s}$ (sans avoir recours explicitement à la fonction $F$). Ainsi, si on sait inverser la transformée de Fourier sur $H$, alors on peut ainsi recouvrer les $f_s$, donc $F$.

% subsection décomposition_en_classes_modulo_h_ (end)

\subsection{Catégorie engendrée} % (fold)
\label{sub:categorie_engendree}

% \begin{defn}[stabilité par somme directe et sous-quotient]
% 	On dit que la catégorie $\cat$ est 
% 	\begin{itemize}
% 		\item \textbf{stable par somme directe} si pour toute famille $(\pi_i)_{i \in I}$ de représentations de $\cat$, la représentation $\bigoplus_{i \in I} \pi_i$ est encore une représentation de $\cat$.
% 		\item \textbf{stable par sous-représentation} si pour toutes représentations $\pi_1 \in \Rep(G)$, $\pi_2 \in \cat$ telle qu'il existe une $G$-injection $\pi_1 \to \pi_2$, alors $\pi_1 \in \cat$.
% 		\item \textbf{stable par quotient} si pour toutes représentations $\pi_1 \in \cat$, $\pi_2 \in \Rep(G)$ telle qu'il existe une $G$-surjection $\pi_1 \to \pi_2$, alors $\pi_2 \in \cat$.
% 		\item \textbf{stable par sous-quotient} si elle est stable par sous-représentation et quotient.
% 	\end{itemize}
% \end{defn}
Dans toute la suite on suppose que $\cat$ est stable par somme directe et sous-quotient.

\begin{defn}
	\label{defn:engendrer}
	Si $\mathcal{D}$ est une sous-catégorie pleine de $\cat$, on dira que $\mathcal{D}$ engendre $\cat$ si $\cat$ est la plus petite sous-catégorie de $\Rep(G)$ contenant $\mathcal{D}$ et stable par somme directe arbitraire et sous-quotient. 
\end{defn}

\begin{ex}
	Considérons $\mathcal{D}$ la catégorie dont la seule représentation est $(\lambda, \Hecke(G))$ où $\lambda$ est la translation à gauche (dont les morphismes sont les endomorphismes de $\Hecke(G)$-module de $\Hecke(G)$). Alors $\mathcal{D}$ engendre $\Rep(G)$. En effet, si $(\pi, V) \in \Rep(G)$ est une représentation quelconque, alors
	\begin{align*}
		\bigoplus_{v \in V} \Hecke(G) &\longrightarrow V \\
		(h_v)_{v \in V} &\longmapsto \sum_{v \in V} h_v.v
	\end{align*}
	est un $G$-morphisme surjectif (car $\pi$ est lisse).
\end{ex}

\begin{lemme}
	\label{lemme:categorie_engendree}
	Si $\mathcal{D}$ engendre $\cat$, alors la restriction $\efo(\cat) \to \efo(\mathcal{D})$ est un morphisme injectif.%, dont l'image est l'ensemble des éléments de $\efo(\mathcal{D})$ qui laissent stables les sous-espaces $G$-stables.
\end{lemme}
\begin{proof}
	Soit $\varphi \in \efo(\cat)$ tel que
	\[
		\forall \pi \in \mathcal{D}, \ \varphi(\pi) = 0
	\]
	Soit $\cat'$ la sous-catégorie pleine de $\cat$ des représentations $\pi \in \cat$ telles que $\varphi(\pi) = 0$. Notre but est de montrer que $\cat' = \cat$. On sait déjà que $\cat'$ contient $\mathcal{D}$, il reste donc à montrer que $\cat'$ est stable par sous-quotient et sommes directes.
	
	Si $i : \pi_1 \to \pi_2$ est une $G$-injection avec $\pi_1 \in \cat$ et $\pi_2 \in \cat'$, la commutation aux $G$-entrelacements entraine $i \circ \varphi(\pi_1) = 0 \circ i = 0$, donc $\varphi(\pi_1) = 0$ (puisque $i$ est injective) et $\pi_1 \in \cat$.
	
	De même, si $s : \pi_1 \to \pi_2$ est une $G$-surjection avec $\pi_1 \in \cat'$ et $\pi_2 \in \cat$, alors $\varphi(\pi_2) \circ s = s \circ 0 = 0$, donc $\varphi(\pi_2) = 0$ (puisque $s$ est surjective) et $\pi_2 \in \cat'$. Ce qui prouve que $\cat'$ est stable par sous-quotient.
	
	Enfin si $(\pi_i)_{i \in I}$ est une famille de représentations de $\cat'$, la commutation aux $G$-entrelacements entraine
	\[
		\varphi\left(\bigoplus_{i \in I} \pi_i \right) = \bigoplus_{i \in I} \varphi(\pi_i) = 0
	\]
	où on donne le sens évident% (voir \ref{sub:somme_directe}) 
	à la somme directe de morphismes. On a donc $\bigoplus_{i \in I} \pi_i \in \cat'$.
	
	En conclusion la sous-catégorie $\cat'$ contient $\mathcal{D}$ et est stable par sous-quotient et somme directe, c'est donc $\cat$ toute entière. Ce qui veut dire que $\varphi$ est nul et prouve l'injectivité.
	
	% Donnons maintenant la caractérisation de l'image. D'abord, il est clair que les éléments de l'image sont des endomorphismes dans $\efo(\mathcal{D})$ qui laissent stables les sous-espaces $G$-stables (ils correspondent à des représentations de $\cat$). Réciproquement, si on se donne un endomorphisme $\varphi \in \efo(\mathcal{D})$ qui laisse stable les sous-espaces $G$-stables alors on peut entendre $\varphi$ sur $\cat$ de manière inductive. Pour commencer $\varphi$ est définie sur $\mathcal{D}$. Si par induction on suppose $\varphi(\pi_2)$ est définie et $i : \pi_1 \to \pi_2$ est une $G$-injection, on définit $\varphi(\pi_1)$ de la manière suivante\,\footnote{Il faut encore vérifier que cette définition ne dépend pas du plongement choisi, ce qui n'est pas évident a priori (et si c'est faux, alors il faut rajouter cette condition à la caractérisation de l'image, ou prendre des restrictions sur la catégorie $\mathcal{D}$).}. Comme $i(\pi_1)$ est stable par $\varphi(\pi_2)$, alors $\varphi(\pi_2) \circ i$ peut s'écrire $i \circ \varphi(\pi_1)$ avec $\varphi(\pi_1) \in \End_{\C}(\pi_1)$. On procède de même pour les quotients. Quant à l'extension aux sommes directes, elle est facile. A chaque étape, on vérifie que par construction, les endomorphismes définis commutent aux entrelacements et laissent stables les sous-espaces stables. Par principe d'induction, cela étend donc $\varphi$ à $\cat$ tout entier.
\end{proof}

\begin{rem}
	L'injectivité de $\efo(\cat) \to \efo(\mathcal{D})$ ne garantit pas que $\mathcal{D}$ engendre $\cat$ (au sens de la définition \ref{defn:engendrer}). Par exemple, si $\Irr(G)$ désigne la catégorie des représentations irréductibles de $G$, alors $\efo(\Rep(G)) \to \efo(\Irr(G))$ est injective (d'après le lemme de séparation et le lemme \ref{lemme:densite}) mais la catégorie engendrée par $\Irr(G)$ (au sens de la définition \ref{defn:engendrer}) est la catégorie des représentations semi-simples, qui est en général strictement plus petite que $\Rep(G)$ (quand $G$ n'est pas compact).
\end{rem}
% subsection catégorie_engendrée (end)

\subsection{Support} % (fold)
\label{sub:support}
Soit $H \subset G$ un sous-groupe distingué de $G$ d'indice fini et fermé. On suppose désormais que si $\pi \in \cat$ et $\rho \in \Rep(G)$ sont telles que $\pi_{|H} \simeq \rho_{|H}$, alors $\rho \in \cat$ (si $G/H$ est commutatif, cela revient à demander que $\cat$ soit stable par torsion par les caractères de $G/H$).
\begin{defn}[Support]
	\label{defn:support}
	On dit que $\varphi \in \efo(\cat)$ est à support dans $H$ si $\varphi$ commute aux $H$-entrelacements, c'est-à-dire si pour tout $\pi_1, \pi_2 \in \cat$ et tout $\alpha \in \Hom_H(\pi_1, \pi_2)$, on a
	\[
		\alpha \circ \varphi(\pi_1) = \varphi(\pi_2) \circ \alpha
	\]
	On note $\efo(\cat)_{|H}$ la sous-algèbre des éléments à support dans $H$. De même, si $g \in G$, on dit que $\varphi$ est à support dans $gH$ si $\widehat{\delta}_{g^{-1}} \varphi$ est à support dans $H$.
\end{defn}

Donnons une caractérisation équivalente de cette condition. On suppose que notre catégorie $\cat$ est telle que si $\pi \in \cat$ et $\rho \in \Rep(G/H)$, alors $\rho \otimes \pi \in \cat$ (ici, on identifie les représentations de $G/H$ et les représentations de $G$ triviales sur $H$). Remarquons que puisque $\cat$ est supposé stable par sommes directes arbitraires et que $\Rep(G/H)$ est semi-simple, il revient au même d'imposer cette condition uniquement pour $\rho \in \Irr(G/H)$

\begin{lemme}
	\label{lemme:caracterisation_tensorielle_support}
	Soit $\varphi \in \efo(\cat)$ et $g \in G$. On a équivalence des propriétés suivantes
	\begin{enumerate}
		\item $\varphi$ est à support dans $gH$
		\item Pour toutes $\pi \in \cat$ et $\rho \in \Rep(G/H)$, alors $\varphi(\rho \otimes \pi) = \rho(g) \otimes \varphi(\pi)$
		\item Pour toutes $\pi \in \cat$ et $\rho \in \Irr(G/H)$, alors $\varphi(\rho \otimes \pi) = \rho(g) \otimes \varphi(\pi)$
 	\end{enumerate}
\end{lemme}
\begin{proof}
	Quitte à remplacer $\varphi$ par $\widehat{\delta}_{g^{-1}} \varphi$, on peut supposer $g = 1$ dans toute la preuve.	L'équivalence de 2 et 3 est évidente (la réciproque découle de la semi-simplicité de $\Rep(G/H)$). Supposons vraie la propriété 1, et montrons 2. Soit $(\rho, W) \in \Rep(G/H)$, et $(\pi,V) \in \cat$. On note $(1_W,W) \in \Rep(G/H)$ la représentation triviale sur $W$. Alors $\Id_{V \otimes W}$ est un $H$-isomorphisme entre $\rho \otimes \pi$ et $1_W \otimes \pi$ (cette dernière représentation est une somme de copies de $\pi$). Donc
	\[
		\varphi(\rho \otimes \pi) = \varphi(1_W \otimes \pi) = \Id_W \otimes \varphi(\pi) = \rho(1) \otimes \varphi(\pi)
	\]
	Réciproquement, supposons 2 et montrons 1. Soient $(\pi_1, V_1), (\pi_2,V_2) \in \cat$. On note $\rho \in \Rep(G/H)$ la représentation dont l'espace vectoriel est $\Hom_H(\pi_1, \pi_2)$ et l'action est $\rho(g).f = \pi_2(g) \circ f \circ \pi_1(g)^{-1}$ pour $f \in \Hom_H(\pi_1, \pi_2)$ et $g \in G$ (c'est bien une représentation de $G/H$ puisque l'action de $H$ est triviale). Définissons un $G$-entrelacement $\alpha : \rho \otimes \pi_1 \longrightarrow \pi_2$ pour $f \in \Hom_H(\pi_1, \pi_2)$ et $v_1 \in V_1$ par
	\[
		\alpha(f \otimes v_1) = f(v_1)
	\]
	Comme $\varphi \in \efo(\cat)$, on a $\alpha \circ \varphi(\rho \otimes \pi_1) = \varphi(\pi_2) \circ \alpha$. Or par hypothèse, on a $\varphi(\rho \otimes \pi_1) = \rho(1) \otimes \varphi(\pi_1)$.	Donc pour tout $f \in Hom_H(\pi_1, \pi_2)$ et $v_1 \in V_1$, on a
	\[
		[\alpha \circ \varphi(\rho \otimes \pi_1)] (f \otimes v_1) = \alpha (f \otimes \varphi(\pi_1)(v_1)) = [f \circ \varphi(\pi_1)](v_1)		\et		[\varphi(\pi_2) \circ \alpha](f \otimes v_1) = [\varphi(\pi_2) \circ f](v_1)
	\]
	D'où on tire
	\[
		f \circ \varphi(\pi_1) = \varphi(\pi_2) \circ f
	\]
\end{proof}

Remarquons que si $G/H$ est commutatif, les représentations irréductibles de $G/H$ sont les caractères et le lemme précédent prend une forme particulièrement simple. On note $X_H = \Hom(G/H, \C^*)$ le groupe de ses caractères que l'on identifie à celui des caractères de $G$ triviaux sur $H$. Le lemme suivant éclaire la définition du support que nous avons donné.

\begin{lemme}
	\label{lemme:support_fonction_transformee}
	Soit $f \in \Hecke(G)$. Alors $\widehat{f}$ est à support dans $H$ si et seulement si $\widehat{f} = \widehat{f . \mathds{1}_{H}}$. En particulier, si $\cat$ est telle que $f \mapsto \widehat{f}$ est injective, alors $\widehat{f}$ est à support dans $H$ si et seulement si $f$ est à support dans $H$.
\end{lemme}
\begin{proof}
	Supposons $\widehat{f} = \widehat{f . \mathds{1}_{H}}$. Si $\pi_1, \pi_2 \in \cat$ et $\alpha \in \Hom_H(\pi_1, \pi_2)$ alors
	\begin{align*}
		\alpha \circ \widehat{f}(\pi_1) &= \alpha \circ \widehat{f \mathds{1}_H}(\pi_1) \\
		&= \alpha \circ \int_H f(h) \, \pi_1(h) \, \D h \\
		&= \int_H f(h) \, \pi_2(h) \circ \alpha \, \D h\\
		&= \widehat{f \mathds{1}_H}(\pi_2) \circ \alpha \\
		&= \widehat{f}(\pi_2) \circ \alpha
	\end{align*}
	Réciproquement, supposons que $\widehat{f}$ est à support dans $H$. D'après le lemme \ref{lemme:caracterisation_tensorielle_support}, pour tout $\rho \in \Irr(G/H)$ et $\pi \in \cat$, on a
	\[
		\widehat{f}(\rho \otimes \pi) = \rho(1) \otimes \widehat{f}(\pi)
	\]
	Et par ailleurs, on calcule
	\[
		\widehat{f}(\rho \otimes \pi) = \sum_{g \in G/H} \int_{H} \rho(g) \otimes f(gh) \, \D h  = \sum_{g \in G/H} \rho(g) \otimes \widehat{f . \mathds{1}_{gH}}(\pi)
	\]
	En notant $\chi_{\rho}$ le caractère de la représentation $\rho$, on tire pour tout $\rho \in \Irr(G/H)$
	\[
		\sum_{g \in G/H} \chi_{\rho}(g) \, \widehat{f . \mathds{1}_{gH}}(\pi) = \chi_{\rho}(1) \, \widehat{f}(\pi)
	\]
	Par linéarité, la relation est encore vraie pour toute combinaison linéaire des $\chi_{\rho}$ (c'est-à-dire pour toute fonction sur $G/H$ invariante par conjugaison). En particulier, pour $\mathds{1}_H$, la fonction caractéristique de $H$, on tire $\widehat{f} = \widehat{f . \mathds{1}_{H}}$.
\end{proof}

\subsection{Le foncteur $\Ind_H^G \circ r_H^G$, avec G/H fini commutatif} % (fold)
\label{sub:resultats_sur_la_restriction_a_un_sous_groupe_d_indice_fini}

Dans cette section, on suppose $G/H$ commutatif (le cas qui nous intéresse réellement), et on spécialise certains résultats à ce cas. On sait (voir par exemple \cite{HenniartCocompact} section 2) que si $\Pi \in \Rep(G)$ est irréductible, alors $\Pi_{|H}$ est semi-simple de longueur finie, et réciproquement, toute représentation irréductible de $H$ apparaît dans la décomposition d'une telle représentation. Par réciprocité de Frobenius, si $\pi \in \Rep(H)$ est irréductible, alors $\Ind_H^G(\pi)$ est semi-simple de longueur finie, et toute représentation irréductible de $G$ apparaît dans la décomposition d'une telle représentation. Le lemme suivant fournit une description des $H$-entrelacements entre des représentations de $G$ dans le cas où $G/H$ est commutatif

\begin{lemme}
	\label{lemme:entrelacement_sur_H}
	Supposons $G/H$ commutatif. Soient $\Pi_1, \Pi_2 \in \Rep(G)$ des représentations lisses de $G$. Alors on a
	\[
		\Hom_H(\Pi_1, \Pi_2) = \bigoplus_{\chi \in X_H} \Hom_G(\chi \Pi_1, \Pi_2)
	\]
	En particulier, quand $\Pi_1$ et $\Pi_2$ sont irréductibles, on obtient l'équivalence des trois propriétés suivantes
	\begin{enumerate}
		\item $(\Pi_1)_{|H}$ et $(\Pi_2)_{|H}$ ont une composante irréductible commune
		\item $(\Pi_1)_{|H} \simeq (\Pi_2)_{|H}$
		\item il existe $\chi \in X_H$ tel quel $\Pi_2 \simeq \chi \Pi_1$
	\end{enumerate}
\end{lemme}
\begin{proof}
	Notons $V = \Hom_H(\Pi_1, \Pi_2)$. On définit une action $\rho$ linéaire de $G$ sur $V$, donnée pour $v \in V$ et $g \in G$ par
	\[
		\rho(g).v = \Pi_2(g) \circ v \circ \Pi_1(g^{-1})
	\]
	Alors on vérifie que $(\rho, V)$ est une représentation de $G$ qui est triviale sur $H$, donc une représentation de $G/H$.  Or comme $G/H$ est fini commutatif, alors $(\rho, V)$ est la somme directe de ses sous-espaces $\chi$-isotypiques pour $\chi \in X_H$. Or pour $\chi \in X_H$, le sous-espace $\chi$-isotypique est justement
	\[
		V_{\chi} = \braces{v \in V, \ \forall g \in G, \  \rho(g).v = \chi(g) v} = \Hom_G(\chi \Pi_1, \Pi_2)
	\]
	Ce qui donne finalement la décomposition
	\[
		\Hom_H(\Pi_1, \Pi_2) = \bigoplus_{\chi \in X_H} \Hom_G(\chi \Pi_1, \Pi_2)
	\]
	Si maintenant on suppose $\Pi_1$ et $\Pi_2$ irréductibles, alors $(\Pi_1)_{|H}$ et $(\Pi_2)_{|H}$ sont semi-simples (car $G/H$ est fini). Si $(\Pi_1)_{|H}$ et $(\Pi_2)_{|H}$ ont une composante irréductible commune, alors, par semi-simplicité, on a $\Hom_H(\Pi_1, \Pi_2) \neq (0)$. D'après la décomposition que l'on vient de montrer, il existe un caractère $\chi \in X_H$, tel que $\Hom_G(\chi \Pi_1, \Pi_2) \neq (0)$. On se donne un tel caractère $\chi$. Comme les représentations $\chi \Pi_1$ et $\Pi_2$ sont irréductibles, un élément non nul de $\Hom_G(\chi \Pi_1, \Pi_2)$ est nécessairement un isomorphisme. D'où $\Pi_2 \simeq \chi \Pi_1$. Réciproquement, il est immédiat que 3 entraine 2 qui entraine 1.
\end{proof}
\begin{rem}
	A titre indicatif, signalons que la projection sur le sous-espace $\chi$-isotypique de $(V, \rho)$ est explicitement donnée par
	\[
		p_{\chi} = \frac{1}{|G/H|}\sum_{g \in G/H} \chi(g)^{-1} \rho(g)
	\]
\end{rem}
En conséquence du lemme \ref{lemme:entrelacement_sur_H}, on peut tirer une décomposition du foncteur $\Ind_H^G \circ r_H^G$.

\begin{lemme}
	\label{lemme:induction_restriction}
	Supposons $G/H$ commutatif. Soit $\Pi \in \Rep(G)$ une représentation lisse de $G$. Alors on a un $G$-isomorphisme naturel en $\Pi$
	
	\[
		\Ind_H^G (\Pi_{|H}) \simeq \bigoplus_{\chi \in X_H} \chi \Pi
 	\]
\end{lemme}
\begin{proof}
	Pour tout $\rho \in \Rep(G)$, par réciprocité de Frobenius, on a un isomorphisme (d'espaces vectoriels) naturel en $\rho$ et $\Pi$, 
	\[
		 \Hom_G(\rho, \Ind_H^G (\Pi_{|H})) \simeq \Hom_H(\rho,\Pi)
	\]
	Et par le lemme \ref{lemme:entrelacement_sur_H}, on a
	\[
		\Hom_H(\rho,\Pi_{|H}) = \bigoplus_{\chi \in X_H} \Hom_G(\rho, \chi \Pi) \simeq \Hom_G\pars{\rho, \bigoplus_{\chi \in X_H} \chi \Pi}
	\]
	D'où l'isomorphisme annoncé grâce au lemme de Yoneda.
\end{proof}
\begin{rem}
	De manière explicite, $\chi \Pi$ s'identifie au sous-espace de $\Ind_H^G (\Pi_{|H})$ des fonctions $w$ vérifiant pour tout $x \in G$
	\[
		w(x) = \chi(x)\Pi(x).w(1)
	\]
	Et la $G$-projection sur cet espace est donnée par
	
	\begin{align*}
		p_{\chi} : \Ind_H^G (\Pi_{|H}) &\longrightarrow \Ind_H^G (\Pi_{|H}) \\
		w &\longmapsto \pars{x \mapsto \frac{1}{|G/H|} \sum_{g \in G/H} \chi(xg^{-1}) \Pi(x).w(g)}
	\end{align*}
\end{rem}

\begin{lemme}
	\label{lemme:restriction_a_H}
	Soit $\pi \in \Rep(H)$ une représentation lisse irréductible de $H$. Alors le groupe $X_H$ des caractères de $G/H$ agit par torsion sur $\text{JH}(\Ind_H^G(\pi))$ de manière transitive. Si on se donne $\Pi \in \text{JH}(\Ind_H^G(\pi))$ et que l'on note $X_H(\pi)$ le stabilisateur $\Pi$ sous cette action et $m(\pi) = \dim \Hom_G(\Pi, \Ind_H^G \pi)$, alors $X_H(\pi)$ et $m(\pi)$ ne dépendent pas du choix de $\Pi$, et on peut écrire
	\[
		\Ind_H^G(\pi) = m(\pi) \, \bigoplus_{\chi \in X_H/X_H(\pi)} \chi \Pi
	\]
\end{lemme}
\begin{proof}
	Si $\chi \in X_H$ est un caractère de $G/H$, la multiplication par $\chi$ fournit un $G$-isomorphisme $[\chi] : \chi \Ind_H^G(\pi) \to \Ind_H^G(\pi)$ donc l'application $f \mapsto f \circ [\chi]$ fournit un isomorphisme d'espaces vectoriels
	\[
		\Hom_{G}(\Pi, \Ind_H^G(\pi)) \simeq \Hom_{G}(\chi\Pi, \Ind_H^G(\pi))
	\]
En particulier, on tire que $X_H$ agit par torsion sur $\text{JH}(\Ind_H^G(\pi))$, et que pour tout $\Pi \in \Rep(G)$, on a $\dim \Hom_G(\Pi, \Ind_H^G \pi) = \dim \Hom_G(\chi \Pi, \Ind_H^G \pi)$.
	
	On se donne $\Pi \in \text{JH}(\Ind_H^G \pi)$ et une $G$-surjection $\Ind_H^G \pi \to \Pi$. Par réciprocité de Frobenius, on tire une $H$-injection $\pi \to \Pi_{|H}$, puis en induisant on obtient une $G$-injection (l'induction est exacte)
	\[
		\Ind_H^G(\pi) \longrightarrow \Ind_H^G(\Pi_{|H}) = \bigoplus_{\chi \in X_H} \chi \Pi
	\]
	Ce qui montre la transitivité de l'action. Comme l'action de $X_H$ est transitive, tous les stabilisateurs des points sont conjugués, et comme le groupe est commutatif, ils sont tous égaux. En notant $X_H(\pi)$ cet unique stabilisateur on peut finalement écrire
	 \[
		\Ind_H^G(\pi) = m(\pi) \, \bigoplus_{\chi \in X_H/X_H(\pi)} \chi \Pi
	\]
\end{proof}
\begin{rem}
	Dans le cas où     $G/H$ est cyclique, alors la multiplicité $m(\pi)$ vaut en fait $1$.%prouve
\end{rem}

Pour finir donnons un critère de commutation aux $H$-entrelacements conséquence du lemme \ref{lemme:entrelacement_sur_H}.
\begin{lemme}
	\label{lemme:critere_support}
	Supposons $G/H$ commutatif. Soit $\varphi \in \efo(\cat)$, alors $\varphi \in \efo(\cat)_{|H}$ si et seulement si pour tout $\pi \in \cat$ et tout $\chi \in X_H$, $\varphi(\chi\pi) = \varphi(\pi)$.
\end{lemme}
\begin{proof}
	Supposons $\varphi \in \efo(\cat)_{|H}$. Soit $\pi \in \cat$ et $\chi \in X_H$, alors l'identité réalise un $H$-entrelacement entre $\pi$ et $\chi \pi$, d'où $\varphi(\pi) = \varphi(\chi \pi)$. Réciproquement, supposons que pour tout $\pi \in \cat$ et tout $\chi \in X_H$, $\varphi(\chi\pi) = \varphi(\pi)$. Soient $\pi_1, \pi_2 \in \cat$, et $\alpha \in \Hom_H(\pi_1, \pi_2)$. D'après le lemme \ref{lemme:entrelacement_sur_H}, on peut écrire $\alpha = \sum_{\chi \in X_H} \alpha_{\chi}$ avec $\alpha_{\chi} \in \Hom_G(\chi \pi_1, \pi_2)$. Pour $\chi \in X_H$, on a
	\[
		\alpha_{\chi} \circ \varphi(\pi_1) = \alpha_{\chi} \circ \varphi(\chi \pi_1) = \varphi(\pi_2) \circ \alpha_{\chi}
	\]
	En sommant, on tire $\alpha \circ \varphi(\pi_1) = \varphi(\pi_2) \circ \alpha$.
\end{proof}
% subsection résultats_sur_la_restriction_à_un_sous_groupe_d_indice_fini (end)

\subsection{Restriction du support} % (fold)
\label{sub:restriction_du_support}

Dans cette section on décrit un procédé pour construire $\widehat{F_{|H}}$ à partir de $\widehat{F}$, et on montre que la construction a bien les propriétés attendues. Commençons par des notations. Pour $\pi \in \Rep(H)$, on pose $\Pi = \Ind_H^G \pi$, et on définit pour $g \in G$
\begin{align*}
	e_{g,\pi} : \Pi &\longrightarrow \pi^g \\
	w & \longmapsto w(g)
\end{align*}
C'est un $H$-entrelacement surjectif qui admet une section
\begin{align*}
	e_{g,\pi}^* : \pi^g &\longrightarrow \Pi \\
	v & \longmapsto x \mapsto \begin{cases}
		\pi(x g^{-1}).v & \textrm{ Si } x \in gH \\
		0 & \textrm{ Sinon}
	\end{cases}
\end{align*}
En l'absence d'ambiguité, on omettra le $\pi$ en indice. On a $e_{g} e_{g}^* = \Id_{\pi^g}$ et $p_{g} = e_{g}^* e_{g}$ est un projecteur, c'est la projection sur les fonctions à support dans $gH$ (en particulier $p_g$ ne dépend que de $gH$). On vérifie également que
\begin{align}
	e_{g_1 g_2, \pi} &= e_{g_2, \pi^{g_1}} \, \lambda(g_1^{-1}) \\
	e_{g_1 g_2, \pi}^* &= \lambda(g_1) \, e^*_{g_2, \pi^{g_1}} \\
	p_{g_1 g_2, \pi} &= \lambda(g_1) \, p_{g_2, \pi^{g_1}} \, \lambda(g_1^{-1})
\end{align}
On a également, en termes de translation à droite, 
\begin{align}
	e_{g_1 g_2, \pi} &= e_{g_1} \, \Pi(g_2) \\
	e_{g_1 g_2, \pi}^* &= \Pi(g_2)^{-1} \, e_{g_1}^* \\
	p_{g_1 g_2, \pi} &= \Pi(g_2)^{-1} p_{g_1} \Pi(g_2) \label{eq:proj_right}
\end{align}
Notons que contrairement aux relations précédentes, les termes qui apparaissent ici ne sont pas des $H$-entrelacements, mais simplement des applications linéaires. On vérifie aussi que
\[
	e_s e_t^* = \begin{cases}
		\pi(st^{-1}) & \text{ Si } s \equiv t \mod H\\
		0 & \text{ Sinon}
	\end{cases} 
\]
Et enfin que
\[
	p_s p_t = \begin{cases}
			p_s & \text{ Si } s \equiv t \mod H\\
			0 & \text{ Sinon}
		\end{cases} \quad \quad \text{ et } \quad \quad \Id_{\Pi} = \sum_{g \in G/H} p_g
\]
On résume ces dernières égalités sous la forme synthétique suivante (que l'on peut voir comme un dual du lemme \ref{lemme:induction_restriction})
\begin{equation}\label{eq:decomp_induite}
	(\Ind_H^G \pi)_{|H} \simeq \bigoplus_{g \in G/H} \pi^g
\end{equation}

\begin{lemme}
	Soient $\rho_1, \rho_2 \in \Rep(H)$ des représentations lisses de $H$, alors on a un isomorphisme naturel en $\pi_1$ et $\pi_2$
	\[
		\Hom_G(\Ind_H^G\pi_1, \Ind_H^G \pi_2) \simeq \bigoplus_{g \in G/H} \Hom_H(\pi_1^g, \pi_2)
	\]
\end{lemme}
\begin{proof}
	C'est une conséquence de l'isomorphisme (\ref{eq:decomp_induite}) plus haut et de la réciprocité de Frobenius.
\end{proof}
\begin{lemme}
	\label{lemme:decomposition_restriction}
	Soit $\Pi \in \Rep(G)$ une représentation lisse irréductible de $G$. Alors $G/H$ agit par conjugaison sur $\text{JH}(\Pi_{|H})$ de manière transitive. Si on fixe $\pi \in \text{JH}(\Pi_{|H})$, que l'on note $G(\pi)$ son stabilisateur sous cette action, et $M(\pi) = \dim Hom_G(\pi, \Pi_{|H})$ (ce qui ne dépend pas du choix de $\pi$), alors on peut écrire
	
	\[
		\Pi_{|H} = M(\pi) \, \bigoplus_{g \in G/G(\pi)} \pi^g
	\]
	
	Si par surcroît on suppose $G/H$ commutatif, alors d'après le lemme \ref{lemme:restriction_a_H} la représentation $\Ind_H^G(\pi)$ se décompose sous la forme
	\[
		\Ind_H^G \pi = m(\pi) \, \bigoplus_{\chi \in X_H / X_H(\pi)} \chi \Pi
	\]
	et l'on a les relations
	\[
		m(\pi) = M(\pi) \quad \text{ et }\quad m(\pi)^2 |X_H / X_H(\pi)| = |G(\pi)/H|
	\]
\end{lemme}
\begin{proof}
	Pour tout $\pi \in \Rep(H)$, on a
	\[
		\Hom_{H}(\pi, \Pi_{|H}) = \Hom_{H}(\pi^g, \Pi_{|H})
	\]
	et $\pi^h \simeq \pi$ si $h \in H$. Donc $G/H$ agit sur conjugaison sur $\text{JH}(\Pi_{|H})$. Soit $\pi \in \text{JH}(\Pi_{|H})$. On se donne une $H$-injection $\pi \to \Pi_{|H}$, puis par réciprocité de Frobenius, on tire une $G$-surjection $\Ind_H^G \pi \to \Pi$ et par restriction on trouve une $G$-surjection (la restriction est exacte)
	\[
		(\Ind_H^G \pi)_{|H} = \bigoplus_{g \in G/H} \pi^g \longrightarrow \Pi_{|H}
	\]
	Ce qui prouve la transitivité. On peut donc écrire
	\[
		\Pi_{|H} = M(\pi) \, \bigoplus_{g \in G/G(\pi)} \pi^g
	\]
	Puisque $\pi \in \text{JH}(\Pi_{|H})$, alors $\Pi \in \text{JH}(\Ind_H^G \pi)$ (par réciprocité de Frobenius et semi-simplicité). Si $G/H$ est commutatif, alors on peut appliquer le lemme  \ref{lemme:restriction_a_H}, ce qui nous donne
	\[
		\Ind_H^G \pi = m(\pi) \, \bigoplus_{\chi \in X_H / X_H(\pi)} \chi \Pi
	\]
	La relation $m(\pi) = M(\pi)$ est conséquence de la réciprocité de Frobenius, pour la seconde on calcule
	\[
		\dim \End_G(\Ind_H^G \pi) = \dim \End_G\pars{m(\pi) \, \bigoplus_{\chi \in X_H / X_H(\pi)} \chi \Pi} = m(\pi)^2 |X_H/X_H(\pi)|
	\]
	Mais par réciprocité de Frobenius, on peut aussi écrire
	\begin{align*}
		\dim \End_G(\Ind_H^G \pi) &= \dim \Hom_H(\pi, (\Ind_H^G \pi)_{|H})\\
		& = \sum_{g \in G/H} \dim \Hom_H(\pi, \pi^g)\\
		&= |G(\pi)/H|
	\end{align*}
\end{proof}

On note $\cat_H$ la sous-catégorie pleine de $\Rep(H)$ engendrée par les $r_H^G(\pi) = \pi_{|H}$ pour $\pi \in \cat$. On remarque que la condition imposée sur $\cat$ en \ref{sub:support} assure que si $\rho \in \cat_H$, alors $\Ind_H^G\rho \in \cat$ et ces représentations engendrent $\cat$ (puisque $\Ind_H^G(\Pi_{|H})$ contient toujours $\Pi$).

\begin{defn}[Restriction à $H$]
	\label{defn:restriction_a_H}
	Si $\varphi \in \efo(\cat)$, on définit $r_H^G(\varphi)$ pour $\pi \in \cat_H$

	\begin{equation}
		r_H^G(\varphi)(\pi) = e_{1,\pi} \, \varphi(\Ind_H^G \pi) \, e^*_{1,\pi}
	\end{equation}
	Pour alléger les notations, on notera souvent $\varphi_{|H}$ à la place de $r_H^G(\varphi)$.
\end{defn}

Vérifions pour commencer que cela définit un élément de $\efo(\cat_H)$.
\begin{lemme}[$\varphi_{|H}$ commute aux entrelacements]
	\label{lemme:restriction_commute}
	Si $\varphi \in \efo(\cat)$, alors $\varphi_{|H} \in \efo(\cat_H)$. Et si $\varphi$ est lisse, alors $\varphi_{|H}$ aussi.
\end{lemme}
\begin{proof}
	Soient $\rho_1, \rho_2 \in \cat_H$ et $\alpha \in \Hom_H(\rho_1, \rho_2)$. Montrons que $\alpha \varphi_{|H}(\pi_1) = \varphi_{|H}(\pi_2) \alpha$. On pose $\pi_1 = \Ind_H^G \rho_1$, $\pi_2 = \Ind_H^G \rho_2$ et $A = \Ind_H^G \alpha$. On vérifie que
	\[
		\alpha e_{1, \rho_1} = e_{1, \rho_2} A \quad \quad \text{ et } \quad \quad A e_{1, \rho_1}^* = e_{1, \rho_2}^* \alpha
	\]
	Et on a
	\begin{align*}
		\alpha \varphi_{|H}(\rho_1) &= \alpha e_{1, \rho_1} \varphi(\pi_1) e^*_{1, \rho_1} \\
		&= e_{1, \rho_2} A \varphi(\pi_1) e^*_{1, \rho_1} \\
		&= e_{1, \rho_2} \varphi(\pi_2) A e^*_{1, \rho_1} \\
		&= e_{1, \rho_2} \varphi(\pi_2) e^*_{1, \rho_2} \alpha \\
		&= \varphi_{|H}(\rho_2) \alpha
	\end{align*}
	Donc $\varphi_{|H} \in \efo(\cat_H)$. Par ailleurs si $\varphi$ est $K$-bi-invariant, alors $\varphi_{|H}$ est $K \cap H$-bi-invariant (ce qui est un sous-groupe ouvert compact de $H$). En effet, si $k_1, k_2 \in K \cap H$, $\rho \in \cat_H$ et $\pi = \Ind_H^G \rho$, on a
	\[
		\rho(k_1) \varphi_{|H}(\rho) \rho(k_2) = \rho(k_1) e \varphi(\pi) e^* \rho(k_2) = e \pi(k_1) \varphi(\pi) \pi(k_2) e^* = e \varphi(\pi) e^* = \varphi_{|H}(\pi) 
	\]
\end{proof}

Dans le lemme suivant, on vérifie que la construction est bien celle qu'on a annoncée et justifie donc le nom et la notation.
\begin{lemme}[la restriction commute au chapeau]
	\label{lemme:restriction_commute_chapeau}
	Soit $F \in \Hecke(G)$, alors $r_H^G(\widehat{F}) = \widehat{F_{|H}}$ (c'est-à-dire que la restriction commute au chapeau : $\widehat{F}_{|H} = \widehat{F_{|H}}$).
\end{lemme}
\begin{proof}
	Soit $\pi \in \cat_H$ et $\Pi = \Ind_H^G \pi$. On a
	\begin{align*}
		\widehat{F}_{|H}(\pi) &= e_1 \widehat{F}(\Pi) e_1^* \\
		&= \sum_{s \in S} e_1 \Pi(s) \pars{\int_{H} F(sh) \Pi(h) \, \D h}  e_1^* \\
		&= \sum_{s \in S} e_1 \Pi(s) e_1^* \pars{\int_{H} F(sh) \pi(h) \, \D h} \\
		&= \sum_{s \in S} e_1 e_{s^{-1}}^* \pars{\int_{H} F(sh) \pi(h) \, \D h} \\
		&= \widehat{F_{|H}}(\pi)
	\end{align*}
	où on rappelle que l'ensemble $S \subset G$ est un système de représentants de $G/H$.
\end{proof}

On vérifie que la restriction satisfait une propriété de multiplicativité analogue à celle pour les distributions.
\begin{lemme}[Multiplicativité de la restriction]
	Soient $\varphi_1, \varphi_2 \in \efo(\cat)$. Si $\varphi_1 \in \efo(\cat)_{|H}$ ou $\varphi_2 \in \efo(\cat)_{|H}$, alors
	\[
		r_H^G(\varphi_1 \varphi_2) = r_H^G(\varphi_1) \, r_H^G(\varphi_2)
	\]
	En particulier $r_H^G$ définit un morphisme d'algèbres de $\efo(\cat)_{|H}$ dans $\efo(\cat_H)$.
\end{lemme}
\begin{proof}
	Soit $\pi \in \cat_H$, on note $\Pi = \Ind_H^G \pi$. On a
	\begin{align*}
		r_H^G(\varphi_1 \varphi_2)(\pi) = e_1 \, \varphi_1(\Pi) \, \varphi_2(\Pi) \, e_1^*
	\end{align*}
	Si $\varphi_1 \in \efo(\cat)_{|H}$, alors $\varphi_1(\Pi) p_1 = p_1 \varphi_1(\Pi)$ (puisque $p_1 : \Pi \to \Pi$ est un $H$-entrelacement), et comme $e_1 = e_1 p_1$, on a donc
	\begin{align*}
		r_H^G(\varphi_1 \varphi_2)(\pi) &= e_1 p_1 \, \varphi_1(\Pi) \, \varphi_2(\Pi) \, e_1^*\\
		&= e_1 \, \varphi_1(\Pi) \, p_1 \, \varphi_2(\Pi) \, e_1^*\\
		&= e_1 \, \varphi_1(\Pi) \, e_1^* e_1 \, \varphi_2(\Pi) \, e_1^*\\
		&= r_H^G(\varphi_1) \, r_H^G(\varphi_2)
	\end{align*}
	Si $\varphi_2 \in \efo(\cat)_{|H}$, on procède de même.
\end{proof}

Réciproquement, on définit un procédé d'inflation comme suit. Si $\phi \in \efo(\cat_H)$, on définit $\text{Inf}_H^G(\phi) \in \efo(\cat)_{|H}$ pour $\Pi \in \cat$ par
\[
	\text{Inf}_H^G(\phi)(\Pi) = \phi(\Pi_{|H})
\]
Il est immédiat de vérifier que l'on a bien $\text{Inf}_H^G(\phi) \in \efo(\cat)_{|H}$. Par ailleurs, si $f \in \Hecke(H)$ et $F \in \Hecke(G)$ l'unique fonction à support dans $H$ dont la restriction à $H$ est $f$, alors $\text{Inf}_H^G(\widehat{f}) = \widehat{F}$.

\begin{lemme}
	\label{lemme:isomorphisme}
	L'inflation $\text{Inf}_H^G : \efo(\cat_H) \longrightarrow \efo(\cat)_{|H}$ est un isomorphisme d'algèbres dont le réciproque est $r_H^G : \efo(\cat)_{|H} \longrightarrow \efo(\cat_H)$. On a le même résultat sur les parties lisses.
\end{lemme}
\begin{proof}
	Il est immédiat de vérifier que $\text{Inf}_H^G$ est un morphisme d'algèbres. Si $\phi \in \efo(\cat_H)$ et $\pi \in \cat_H$, alors 
	\[
		r_H^G\pars{\text{Inf}_H^G(\phi)}(\pi) = e_1 \phi( (\Ind_H^G \pi)_{|H}) e_1^* = \phi(\pi) e_1 e_1^* = \phi(\pi)
	\]

	Si $\varphi \in \efo(\cat)_{|H}$ et $\Pi \in \cat$, on pose $\pi = \Pi_{|H}$. Alors comme $e_{1, \pi} \in \Hom_H((\Ind_H^G \pi)_{|H}, \Pi_{|H})$ et $\varphi \in \efo(\cat)_{|H}$, on a $\varphi(\Pi) e_1 = e_1 \varphi(\Ind_H^G \pi)$, donc
	
	\[
		\text{Inf}_H^G \pars{r_H^G(\varphi)}(\Pi) = e_{1, \pi} \varphi(\Ind_H^G \pi) e_{1, \pi}^* = \varphi(\Pi) e_1  e_1^* = \varphi(\Pi)
	\]
	Pour finir, notons que si $\phi \in \efo(\cat_H)$ est $K$-bi-invariant, alors $\text{Inf}_H^G(\phi)$ aussi, ce qui donne l'isomorphisme entre les parties lisses (l'autre sens ayant été traité dans le lemme \ref{lemme:restriction_commute}).
\end{proof}
On pose $p_H = \text{Inf}_H^G \circ r_H^G : \efo(\cat) \to \efo(\cat)$. Alors d'après le lemme précédent, c'est un projecteur d'image $\efo(\cat)_{|H}$. Pour $g \in G$ et $\varphi \in \efo(\cat)$, on pose

\begin{equation}
	\varphi_g = p_H(\widehat{\delta}_{g^{-1}} \, \varphi)
\end{equation}
Par construction, $\varphi_g \in \efo(\cat)_H$. Pour tout $h \in H$ on a $p_H(\widehat{\delta}_h \varphi) = \widehat{\delta}_h p_H(\varphi)$, donc $\widehat{\delta}_g \varphi_g$ ne dépend que $gH$. 

\begin{lemme}
	\label{lemme:p_H}
	Si $f \in \Hecke(G)$, alors $p_H(\widehat{f}) = \widehat{f \mathds{1}_H}$
\end{lemme}
\begin{proof}
	Comme $f_{|H} = (f \mathds{1}_H)_{|H}$, alors $r_H^G(\widehat{f}) = r_H^G(\widehat{f \mathds{1}_H})$ (d'après le lemme \ref{lemme:restriction_commute}). Or comme $\Supp(f \mathds{1}_H) \subset H$, alors  $\widehat{f \mathds{1}_H} \in \efo(\cat)_{|H}$ (d'après le sens direct du lemme \ref{lemme:support_fonction_transformee}), et le lemme \ref{lemme:isomorphisme} donne alors

\[
	p_H(\widehat{f}) = \text{Inf}_H^G (r_H^G(\widehat{f \mathds{1}_H})) = \widehat{f \mathds{1}_H}
\]
\end{proof}

\begin{lemme}
	\label{lemme:decomposition}
	Pour $\varphi \in \efo(\cat)$, on a
	\[
		\varphi = \sum_{g \in G/H} \widehat{\delta}_g \varphi_g
	\]
\end{lemme}
\begin{proof}
	D'après le lemme \ref{lemme:categorie_engendree} il suffit de vérifier l'égalité sur des générateurs de $\cat$, en particulier, on peut le faire pour les représentations de type $\Ind_H^G(\pi)$ avec $\pi \in \cat_H$. Soit $\pi \in \cat_H$, on pose $\Pi = \Ind_H^G \pi$. Soit $\varphi \in \efo(\cat)$, on pose $\phi = r_H^G(\varphi)$.
	\begin{align*}
		p_H(\varphi)(\Pi) &= \phi(\Pi_{|H}) = \sum_{s \in S} e_{s}^* \phi(\pi^s) e_{s}\\
		&= \sum_{s \in S} e_{s}^* \, e_{1} \, \varphi(\Ind_H^G \pi^s) \, e_{1}^* \, e_{s} \\
		&= \sum_{s \in S} e_{s}^* \, e_{1} \, \lambda(s^{-1}) \varphi(\Pi) \, \lambda(s) \, e_{1}^* \, e_{s} \\
		&= \sum_{s \in S} p_s \varphi(\Pi) p_s
	\end{align*}
	Donc on obtient, d'après l'équation \ref{eq:proj_right},
	\[
		(\widehat{\delta}_s \varphi_s)(\Pi) = \sum_{t \in G/H} \Pi(s) p_t \Pi(s)^{-1} \varphi(\Pi) p_t = \sum_{t \in G/H} p_{t s^{-1}} \varphi(\Pi) p_t
	\]
	Et finalement, en sommant
	\begin{align*}
		\sum_{s \in G/H} (\widehat{\delta}_s \varphi_s)(\Pi) &= \sum_{s \in G/H} \sum_{t \in G/H} p_{t s^{-1}} \varphi(\Pi) p_t \\
		&\underset{_{u = t s^{-1}}}{=} \sum_{s \in G/H} \sum_{u \in G/H} p_{u} \varphi(\Pi) p_{s} \\
		&= \pars{\sum_{u \in G/H} p_u} \varphi(\Pi) \pars{\sum_{s \in G/H} p_s} \\
		&= \varphi(\Pi)
	\end{align*}
\end{proof}
\begin{rem}
	Dans le cas où $G/H$ est commutatif, on peut donner une preuve alternative. En tirant parti de la décomposition du lemme \ref{lemme:induction_restriction}, on trouve que pour toute représentation $\Pi \in \cat$ (pas uniquement les induites), on a
	\begin{equation}
		\varphi_{g}(\Pi) = \frac{1}{|X_H|}\sum_{\chi \in X_H} (\chi \Pi)(g^{-1}) \varphi(\chi \Pi)
	\end{equation}
	Il est clair grâce au lemme \ref{lemme:critere_support} que $\varphi_g \in \efo(\cat)_{|H}$, et on vérifie aisément que
	\[
		\sum_{g \in G/H} \Pi(g) \varphi_{g}(\Pi) = \frac{1}{|X_H|} \sum_{g \in G/H} \sum_{\chi \in X_H} \chi(g^{-1}) \varphi(\chi \Pi) = \varphi(\Pi)
	\]
\end{rem}

Concluons par ce lemme qui permet de déduire une formule d'inversion sur $G$ de celle sur $H$.

\begin{lemme}
	\label{lemme:conclusion}
	Soit $\varphi \in \efo(\cat)$. Les 3 conditions suivantes sont équivalentes.
	\begin{enumerate}
		\item Il existe $F \in \Hecke(G)$ tel que $\varphi = \widehat{F}$
		\item Pour tout $s \in S$, il existe $F_s \in \Hecke(G)$ tel que $\Supp(F_s) \subset H$ et $\varphi_s = \widehat{F_s}$
		\item Pour tout $s \in S$, il existe $f_s \in \Hecke(H)$ tel que $(\widehat{\delta}_{s^{-1}} \varphi)_{|H} = \widehat{f_s}$
	\end{enumerate}
	Si ces conditions sont vérifiées, alors
	\[
		F = \sum_{s \in S} \delta_s * F_s \quad \quad F_s = (\delta_{s^{-1}} * F) . \mathds{1}_{H} \quad \quad f_s = (F_s)_{|H}
	\]
	C'est-à-dire que $f_s(h) = F(sh)$ pour tout $s \in S$, $h \in H$. 
\end{lemme}
\begin{proof}
	Il s'agit d'une compilation des calculs effectués précédemment. Si $\varphi = \widehat{F}$, on pose $F_s = (\delta_{s^{-1}} * F) . \mathds{1}_{H}$ et $f_s = (F_s)_{|H}$. Il est immédiat que 
	\[
		F = \sum_{s \in S} \delta_s * F_s
	\]
	Et pour $s \in S$ on a
	\[
		\varphi_s = p_H(\widehat{\delta}_{s^{-1}} \, \varphi) = p_H(\widehat{\delta_{s^{-1}} * f}) = \widehat{F_s}
	\]
	Et enfin
	\[
		(\widehat{\delta}_{s^{-1}} \varphi)_{|H} = \widehat{F_s}_{|H} = \widehat{f_s} 
	\]
	Réciproquement supposons que pour tout $s \in S$, il existe $f_s \in \Hecke(H)$ tel que $(\widehat{\delta}_{s^{-1}} \varphi)_{|H} = \widehat{f_s}$. Soit $F_s \in \Hecke(G)$ la fonction à support dans $H$ telle que $(F_s)_{|H} = f_s$, on pose $F = \sum_{s \in S} \delta_s * F_s$. Alors d'après le lemme \ref{lemme:isomorphisme}, on a
	\[
		\widehat{F_s} = \text{Inf}_H^G(\widehat{f_s}) = \varphi_s
	\]
	Enfin, d'après le lemme \ref{lemme:decomposition}, on a
	\[
		\varphi = \sum_{s \in S} \widehat{\delta}_s \varphi_s = \widehat{F}
	\]
\end{proof}
% subsection restriction_du_support (end)

% section endomorphismes_du_foncteur_d_oubli (end)

\section{Induction parabolique dans un groupe réductif non connexe} % (fold)
\label{sec:induction_parabolique_dans_le_groupe_g}

Dans cette partie, $G$ est (le groupe des points sur $F$ d') un groupe réductif défini sur $F$. Muni de la topologie $p$-adique\,\footnote{obtenue en choisissant un plongement quelconque de $G$ dans $\GL_n(F)$ (sachant que la topologie est en fait indépendante du choix de plongement)}, c'est un groupe localement profini, on dispose donc déjà des résultats généraux sur leurs représentations (voir par exemple \cite{Bernstein} chapitre I, et \cite{Renard} I à IV). D'autres propriétés sont reliées à la géométrie des groupes réductifs et diffèrent donc du cas connexe. Dans cette partie, nous rappelons donc succinctement les résultats de l'article \cite{Goldberg} sur l'induction parabolique dans un groupe réductif non connexe. Nous y renvoyons le lecteur pour plus de détails.

\subsection{Levi cuspidaux} % (fold)
\label{sub:levi_et_paraboliques_cuspidaux}

\subsubsection{Propriétés de base} % (fold)
\label{ssub:proprietes_de_base}

% Définissons tout d'abord les notions de sous-groupes de Levi cuspidaux.
Pour un groupe réductif donné $G$, on appellera \textbf{composante déployée} le tore maximal déployé (sur $F$) dans le centre du groupe, que l'on notera $A_G$ (cela s'applique en particulier aux sous-groupes de $G$ qui sont réductifs). 

\begin{defn}[Tore spécial]
	\label{def:tore_special}
	Un tore $A$ déployé sur $F$ est dit \textbf{spécial} dans $G$ si c'est la composante déployée de $C_G(A)$, son centralisateur dans $G$.
\end{defn}

Les tores spéciaux dans $G^0$ sont aussi spéciaux dans $G$ (\cite{Goldberg} lemme 2.1), mais la réciproque n'est pas vraie en général (\cite{Goldberg} remarque 2.2).

\begin{defn}[Levi cuspidal]
	\label{def:levi_cuspidal}
	Les sous-groupes de \textbf{Levi cuspidaux} de $G$ sont les centralisateurs dans $G$ des tores spéciaux de $G^0$. On notera $\mathcal{L}(G)$ l'ensemble des sous-groupes de Levi cuspidaux de $G$.
\end{defn}

% Notons que si $M$ est un sous-groupe de Levi cuspidal, alors $M$ est encore réductif et le groupe des composantes $M/M^0$ est encore commutatif (puisqu'il s'injecte dans $G/G^0$).
%On notera que les tores spéciaux de $G^0$ sont encore spéciaux dans $G$, mais réciproquement les tores spéciaux de $G$ ne le sont pas forcément dans $G^0$. 
Quand $G$ est connexe, on retrouve par cette définition les sous-groupes de Levi usuels. Dans le cas non connexe, on trouve une famille de groupes réductifs en correspondance bijective avec les sous-groupes de Levi de $G^0$ :

\begin{prop}%[voir \cite{Goldberg}]
	L'application ``composante neutre''
	\begin{align*}
		\mathcal{L}(G) & \longrightarrow \mathcal{L}(G^0)\\
		M &\longmapsto M^0 = M \cap G^0
	\end{align*}
	établit une bijection entre les sous-groupes de Levi cuspidaux de $G$ et les sous-groupes de Levi de $G^0$, dont la réciproque est donnée par $M^0  \longmapsto M = C_G(A_{M^0})$. On dit que $M$ est l'unique Levi cuspidal \textbf{au-dessus de} $M^0$.
\end{prop}
\begin{proof}
	Voir \cite{Goldberg} proposition 2.10.
\end{proof}
\begin{rem}
	\label{rem:croissance_composante_connexe}
	En particulier, remarquons que la bijection entre Levi de $G^0$ et Levi cuspidaux de $G$ est croissante pour l'inclusion.
\end{rem}
\begin{rem}
	On remarque que $G$ lui-même peut ne pas être cuspidal. Par exemple, si on considère $G = G^0 \rtimes \scal{\epsilon}$ où $G^0 = F^*$ et $\epsilon$ agit sur $F^*$ par $x \mapsto x^{-1}$, alors l'unique Levi cuspidal au dessus de $G^0$ est $G^0$. 
\end{rem}

% subsubsection propriétés_de_base (end)

\subsubsection{Le cas $G = G^0 \rtimes \langle \epsilon \rangle$} % (fold)
\label{ssub:le_cas_}

%%%%%%%%%%%%%
%  A finir  %
%%%%%%%%%%%%%

Considérons un instant le cas $G = G^0 \rtimes \langle \epsilon \rangle$ avec $\epsilon$ d'ordre $n$. Le lemme suivant décrit les sous-groupes de Levi cuspidaux de $G$.

\begin{lemme}
	Soit $M$ un sous-groupe de Levi cuspidal de $G^0 \rtimes \langle \epsilon \rangle$, alors il existe $g_0 \in G^0$ et un entier $d$ divisant $n$ tels que
	\[
		M = M^0 \rtimes \langle g_0\epsilon^{d} \rangle
	\]
\end{lemme}
\begin{proof}
	Comme $M/M^0$ s'injecte dans $G/G^0 = \langle \epsilon \rangle$, alors il existe $d$ divisant $n$ tel que $M/M^0$ est isomorphe à $\scal{\epsilon^d}$. On se donne $\theta \in M$ un relèvement d'un générateur de $\epsilon^d$, on peut écrire $\theta = g_0 \epsilon^d$ avec $g_0 \in G^0$. Alors on vérifie que $M = M^0 \rtimes \scal{\theta}$.
\end{proof}

\begin{ex}
	Soit $G^0 = \GL_n$ et $G = G^0 \rtimes \langle \epsilon \rangle$ où $\epsilon : g \mapsto (g^{-1})^t$. Si $n_1 + \ldots + n_s = n$ est une partition de $n$ et $M^0 = \GL_{n_1} \times \ldots \times \GL_{n_s}$, alors le sous-groupe de Levi cuspidal au-dessus de $M^0$ est $M^0$.
\end{ex}
	
% subsubsection le_cas_ (end)

\subsubsection{Caractères non ramifiés} % (fold)
\label{ssub:composante_deployee}

Dans cette partie on rappelle quelques résultats sur les caractères d'un groupe cuspidal non connexe (on reproduit certaines preuves de \cite{Goldberg}). Dans la suite, $M \in \mathcal{L}(G)$ désigne un sous-groupe de Levi cuspidal de $G$. On note $M^1$ le sous-groupe de $M^0$ engendré par les sous-groupes compacts de $M^0$. Pour $H$ un groupe algébrique, on note $\Rat(H)$ le groupe de ses caractères algébriques définis sur $F$. On note $a_{M}^* = \Rat(A_M) \otimes \R$, $a_{M, \C}^* = \Rat(A_M) \otimes \C$ et $a_{M}$, $a_{M, \C}$ pour leurs duaux algébriques. 

Un caractère complexe de $M$ sera dit non ramifié s'il est trivial sur $M^1$, et on dira qu'il provient d'un caractère rationnel de $M$ s'il est dans l'image du morphisme $\Rat(M) \otimes \C \to \Hom(M, \C^*)$ qui à $\chi \otimes s$ associe le caractère $m \mapsto |\chi(m)|_F^s$.

Dans le cas où $M$ est connexe, un caractère complexe est non ramifié si et seulement s'il provient d'un caractère rationnel par le procédé précédent. Mais dans le cas non connexe, ces deux notions divergent (un caractère non ramifié de $M/M^0$ ne provient pas d'un caractère rationnel). Dans le cas cuspidal, comme $M$ a la même composante déployée que sa composante connexe $M^0$ (c'est une conséquence de \cite{Goldberg} lemme 2.1) un caractère non ramifié de $M^0$ se prolonge à $M$ en un caractère provenant d'un caractère rationnel de $M$. Rappelons ces résultats en détail.
% On peut alors montrer que l'on peut  prolonger à $M$ les caractères non ramifiés de $M^0$ (voir \cite{Goldberg} lemme 5.4, et définitions) de manière à ce que l'ensemble des prolongements forme un groupe (si $\chi \in X(M^0)$, on notera $\tilde{\chi}$ ce prolongement à $M$). On appellera donc caractères non ramifiés de $M$ le groupe des caractères ainsi formés (voir \cite{Goldberg} pour le prolongement choisi), que l'on notera $X(M)$.

\begin{defn}
	\label{def:caracters_non_ramifies}
	On note $X^0(M)$ l'image du morphisme $\Rat(M) \otimes \C \to \Hom(M, \C^*)$ qui à $\chi \otimes s$ associe le caractère $m \mapsto |\chi(m)|_F^s$. On dit que les caractères de $X^0(M)$ \textbf{proviennent d'un caractère rationnel} de $M$.
	
	On définit le groupe des \textbf{caractères non ramifiés} de $M$, noté $X(M)$ par $X(M) = \Hom(M/M^1, \C^*)$.
	
	On note enfin $X_0(M) = \Hom(M/M^0, \C^*)$ le groupe des caractères du quotient $M/M^0$.
\end{defn}

\begin{lemme}
	\label{lemme:suite_exacte}
	La restriction $\text{Rat}(M) \overset{r}{\longrightarrow} \text{Rat}(A_M)$ a un noyau et conoyau finis.
\end{lemme}
\begin{proof}
	De la suite exacte courte $1 \to M^0 \to M \to M/M^0 \to 1$, on tire la suite exacte
	\[
		1 \longrightarrow \Rat(M/M^0) \overset{i}{\longrightarrow} \Rat(M) \overset{r'}{\longrightarrow} \Rat(M^0)
	\]
	Par ailleurs on sait que la restriction $\Rat(M^0) \overset{r_0}{\longrightarrow} \Rat(A_M)$ est injective (\cite{Silberger} lemme 0.4.1). En composant, on tire l'exactitude de
	\[
		1 \longrightarrow \text{Rat}(M/M^0) \overset{i}{\longrightarrow} \text{Rat}(M) \overset{r}{\longrightarrow} \text{Rat}(A_M)
	\]
	En particulier, le noyau de $r$ est le groupe fini $\text{Rat}(M/M^0)$. Il reste à vérifier que $\coker(r)$ est fini. On choisit un plongement de $G$ dans $\GL_n$ (c'est possible puisque $G$ est linéaire). On pose $\widetilde{M} = C_{\GL_n}(A_M)$, alors $\widetilde{M}$ est un sous-groupe de Levi de $\GL_n$, notons $A$ sa composante déployée. La restriction $\Rat(\widetilde{M}) \overset{f}{\longrightarrow} \Rat(A)$ a un conoyau fini (toujours d'après \cite{Silberger} lemme 0.4.1) et la restriction $\Rat(A) \overset{g}{\longrightarrow} \Rat(A_M)$ est surjective (car $A$ et $A_M$ sont commutatifs), donc $\Rat(\widetilde{M}) \longrightarrow \Rat(A_M)$ a un conoyau fini (c'est l'image par $g$ du conoyau de $f$). Or $M = C_G(A_M) \subset \widetilde{M}$, donc $r$ a aussi un conoyau fini.
	% Comme $A_M$ est un tore déployé, on peut diagonaliser son action sur $V$. Pour $\chi \in \Rat(A_M)$, on pose $V(\chi) = \braces{v \in V, \forall a \in A, a.v = \chi(a)v}$. Soient $\chi_1, \ldots, \chi_n$ les caractères tels que $V(\chi_i) \ne 0$. On note $V_i = V(\chi_i)$, $a_i = \dim V_i$ et $a = \prod_{i = 1}^n a_i$ et $\widetilde{M} = \prod_{i = 1}^n \GL(V_i)$. On a $M = C_G(A_M) \subset C_{\GL(V)}(A_M) = \widetilde{M}$
\end{proof}
\begin{lemme}
	\label{lemme:isomorphisme_algebre_de_lie}
	Si $K$ est un corps commutatif de caractéristique nulle, alors la restriction $\Rat(M) \overset{r}{\longrightarrow} \Rat(A_M)$ induit un isomorphisme de $K$-espaces vectoriels
	\[
		\Rat(M) \otimes K \underset{r}{\simeq} \Rat(A_M) \otimes K
	\]
	En particulier, on a
	\[
		\Rat(M) \otimes \R \simeq a^*_{M^0} \quad \quad \text{ et } \quad \Rat(M) \otimes \C \simeq a^*_{M^0, \C}
	\]
\end{lemme}
\begin{proof}
	D'après le lemme \ref{lemme:suite_exacte}, on dispose d'une suite exacte
	\[
		\Rat(M/M^0) \overset{i}{\longrightarrow} \Rat(M) \overset{r}{\longrightarrow} \Rat(A_M) \longrightarrow \coker(r)
	\]
	Si $K$ est un corps commutatif de caractéristique nulle, on tire en tensorisant par $K$ (qui est plat en tant que $\Z$-module)
	\[
		\Rat(M/M^0) \otimes K \longrightarrow \Rat(M) \otimes K \overset{r}{\longrightarrow} \Rat(A_M) \otimes K \longrightarrow \coker(r) \otimes K
	\]
	Mais comme $\Rat(M/M^0)$ et $\coker(r)$ sont finis, alors les termes extrémaux sont nuls, donc $r$ est un isomorphisme.
\end{proof}
On dispose d'un morphisme
\[
	f_M : \Rat(M) \otimes \C \longrightarrow \Hom(M, \C^*)
\]
qui à $\chi \otimes s$ associe le caractère additif $m \mapsto |\chi(m)|_F^s$. On définit $\psi_M : a_{M, \C}^* \longrightarrow \Hom(M, \C^*)$ par $\psi_M = f_M \circ r^{-1}$. On note $H_M : M \to a_{M}$ l'application définie par $\scal{H_M(m), \chi} = v_F(\chi(m))$ pour $m \in M$ et $\chi \in \Rat(M)$.
% 
% Pour $M^0$, on retrouve la définition usuelle des caractères non ramifiés $X(M^0) \simeq \Hom(M^0/M^1, \C^*)$. Pour $M$ non connexe, le groupe que l'on trouve est strictement plus petit que $\Hom(M^0/M^1, \C^*)$ et en fait isomorphe à $X(M^0)$.
\begin{prop}[cf. \cite{Goldberg} lemma 5.4]
	\label{prop:prolongement_caracteres_non_ramifies}
	Pour tout $\nu \in a_{M, \C}^*$, on a 
	\[
		\psi_M(\nu)_{|M^0} = \psi_{M^0}(\nu)
	\]
	et la restriction $X^0(M) \longrightarrow X(M^0)$ est un isomorphisme.
\end{prop}
\begin{proof}
	Si on examine le diagramme suivant (en reprenant les notations des lemmes précédents, c'est-à-dire que les morphismes $r$ sont donnés par les restrictions )
	\[
		\xymatrix{
		&& \Rat(M) \otimes \C \ar[rr]^{f_M} \ar@/^/[lld]^{r} \ar[dd]^{r'} & & \Hom(M, \C^*) \ar[dd]^{\text{res}_{M^0}^M} \\
		a_{M, \C}^* \ar@/^/[rru]^{r^{-1}} \ar@/_/[rrd]_{r_0^{-1}}\\
		&& \Rat(M^0) \otimes \C \ar[rr]_{f_{M^0}} \ar@/_/[llu]_{r_0} & & \Hom(M^0, \C^*)
		}
	\]
	alors on voit facilement que le triangle et le carré sont commutatifs, donc le diagramme entier l'est. En particulier, la restriction $\Hom(M, \C^*) \longrightarrow \Hom(M^0, \C^*)$ induit effectivement un morphisme $X^0(M) \longrightarrow X(M^0)$. La surjectivité est claire il reste à prouver l'injectivité. Si $\chi \in X^0(M)$ est trivial sur $M^0$, alors il définit un caractère du quotient $M/M^0$, mais un élément de $X^0(M)$ est trivial sur les éléments d'ordre fini (puisque les racines de l'unité sont de norme $1$) donc $\chi = 1$.
\end{proof}

En somme, un caractère non ramifié de $M^0$ se prolonge de manière unique en un caractère de $X^0(M)$. Dorénavant, on se permet donc d'identifier $X(M^0)$ et $X^0(M)$, en particulier si $\chi \in X(M^0) \simeq \Hom(M^0/M^1, \C^*)$, on notera encore $\chi \in X^0(M)$ le prolongement à $M$ obtenu via la proposition \ref{prop:prolongement_caracteres_non_ramifies}. Et si $\chi \in X(M)$ est un caractère de $M$, on notera $\chi_0 = \chi_{|M^0}$ sa restriction à $M^0$. On termine cette section sur les caractères de $M$ par le résultat suivant de décomposition.
\begin{prop}
	\label{prop:caracteres_de_M}
	On a la décomposition
	\[
		X(M) = X^0(M) \times X_0(M)
	\]
\end{prop}
\begin{proof}
	On a une suite exacte de groupes abéliens
	\[
		1 \longrightarrow \Hom(M/M^0, \C^*) \longrightarrow \Hom(M/M^1, \C^*) \longrightarrow \Hom(M^0/M^1, \C^*) \longrightarrow 1
	\]
	qui est scindée à droite en vertu de la proposition \ref{prop:prolongement_caracteres_non_ramifies}.
\end{proof}
\begin{ex}
	Supposons que $M = M^0 \rtimes \langle \epsilon \rangle$ soit cuspidal avec $\epsilon$ un automorphisme extérieur de $G^0$ d'ordre fini. Alors un caractère non ramifié de $M^0$ est invariant par $\epsilon$ (puisque d'après la proposition \ref{prop:prolongement_caracteres_non_ramifies} il est prolongeable à $M$). Pour l'étendre en un caractère de $M$, il suffit de choisir sa valeur en $\epsilon$ (arbitraire parmi les racines $n^{\textnormal{ièmes}}$ de l'unité), ce qui revient à choisir un caractère de $M/M^0 = \scal{\epsilon}$. Le prolongement donné par la proposition \ref{prop:prolongement_caracteres_non_ramifies} est celui qui est trivial en $\epsilon$. % Ainsi le groupe $X(M)$ correspond ici, comme dans le cas connexe, au groupe des caractères triviaux sur tous les sous-groupes compacts de $M$ (si $\chi \in X(M)$ et si $K \subset \cdot M$ est un sous-groupe compact de $M$, alors \chi est trivial sur $K \cap M^0$, et ).
\end{ex}

%%%% Rajout de l'ensemble des caractères dont la restriction à M° est non ramifiée.

% subsubsection composante_déployée (end)

% subsection levi_et_paraboliques_cuspidaux (end)

\subsection{Paraboliques cuspidaux} % (fold)
\label{sub:paraboliques_cuspidaux}

\begin{defn}[parabolique cuspidal]
	Les sous-groupes \textbf{paraboliques cuspidaux} de $G$ sont les sous-groupes de la forme $P = M N$ où $M$ est un sous-groupe de Levi cuspidal de $G$, et $P^0 = M^0 N$ un sous-groupe parabolique de $G^0$. On notera $\mathcal{P}(G)$ l'ensemble des sous-groupes paraboliques cuspidaux de $G$.
\end{defn}
Si $P = MN$ est un sous-groupe parabolique cuspidal de $G$, alors $M$ normalise $N$. Comme pour les sous-groupes de Levi, les sous-groupes paraboliques cuspidaux de $G$ sont en bijection avec les paraboliques de $G^0$ via la composante neutre (voir \cite{Goldberg} proposition 2.10). Pour $M$ un Levi cuspidal de $G$, on note $\mathcal{P}(M)$ l'ensemble des sous-groupes paraboliques cuspidaux de $G$ de composante de Levi $M$. On a donc $|\mathcal{P}(M)| = |\mathcal{P}(M^0)|$. Signalons que les groupes paraboliques cuspidaux ne sont en général pas égaux à leur normalisateur dans $G$.

\subsubsection{Parabolique opposé} % (fold)
\label{ssub:parabolique_oppose}

\begin{defn}[parabolique opposé]
	Soit $P = MN$ est un sous-groupe parabolique cuspidal, on appelle \textbf{parabolique opposé} le groupe $\overline{P} = M \overline{N}$ où $\overline{N}$ est opposé à $N$ au sens usuel.
\end{defn}
Il est clair que $\overline{P}$ est l'unique sous-groupe parabolique cuspidal au dessus de $\overline{P^0}$ et que l'on a $\overline{P} \cap P = M$.
% subsubsection parabolique_opposé (end)

\subsubsection{Décomposition d'Iwasawa} % (fold)
\label{ssub:decomposition_d_iwasawa}

Un mot sur la décomposition d'Iwasawa. L'existence d'un ``bon'' sous-groupe compact maximal dans $G$, n'est pas garantie a priori. Néanmoins, si $P$ est un sous-groupe parabolique cuspidal de $G$, alors le quotient $G/P$ est compact car il s'injecte dans $G/P^0$, qui est en bijection avec $G/G^0 \times G^0/P^0$.  Cela assure de bonnes propriétés pour l'induction que l'on donne dans la section suivante. Enfin si $G$ est le produit semi-direct de $G^0$ par un groupe fini $H$, et si $K_0$ est un sous-groupe compact maximal de $G^0$ stable par $H$-conjugaison, alors le produit semi-direct $K_0 \rtimes H$ fournit une ``décomposition d'Iwasawa''. Par exemple, si $G = \GL_n(F) \rtimes \scal{\theta}$ avec $\theta(g) = g^{-t}$, on vérifie que $K_0 = \GL_n(\mathcal{O}_F)$ est stable par $\theta$, donc on peut considérer $K_0 \rtimes \scal{\theta}$.

% subsubsection décomposition_d_iwasawa (end)

% subsection paraboliques_cuspidaux (end)

\subsection{Induction parabolique, foncteur de Jacquet} % (fold)
\label{sub:induction_parabolique}

Si $P = M N$ est un sous-groupe parabolique cuspidal de $G$ de composante de Levi $M$, on définit les foncteurs d'induction parabolique $\Ind_P^G : \Rep(M) \to \Rep(G)$ et de Jacquet $r_P^G : \Rep(G) \to \Rep(M)$ de manière identique au cas connexe (le facteur de normalisation ne dépend que du radical unipotent de $P$, qui est contenu dans $G^0$). L'induction est l'adjoint à droite de $r_P^G$, les foncteurs sont exacts, l'induction préserve l'admissibilité et le foncteur de Jacquet préserve le type fini (preuves identiques au cas connexe).

% subsection induction_parabolique (end)

\subsection{Représentations cuspidales} % (fold)
\label{sub:representations_cuspidales}

\begin{defn}[représentation cuspidale]
	Si $\pi \in \Rep(G)$ est une représentation irréductible lisse de $G$, on dit que $\pi$ est \textbf{cuspidale} si ses coefficients matriciels sont à support compacts modulo le centre.
\end{defn}
Si $M$ est un sous-groupe de Levi cuspidal de $G$, alors $M$ admet des représentations cuspidales (c'est en fait dans ce but que ces groupes sont construits, et de là que vient la terminologie) On suppose donc dans cette partie que $M$ est cuspidal. Une représentation irréductible de $M$ est cuspidale si et seulement si les composantes irréductibles (en fait, une seule suffit) de sa restriction à $M^0$ sont cuspidales au sens usuel (voir \cite{Goldberg} lemme 1.1). Comme dans le cas connexe, cela équivaut à annuler tous les foncteurs de Jacquet stricts (\cite{Goldberg} proposition 2.17). Un résultat important de \cite{Goldberg} est le suivant (ce résultat assure l'injectivité de la transformée de Fourier).

\begin{thm}%[voir \cite{Goldberg}]
	Soit $\pi \in \Rep(G)$ une représentation irréductible de $G$. Alors il existe un sous-groupe parabolique cuspidal $P = M N$ de $G$ et une représentation $\sigma \in \Rep(M)$ irréductible cuspidale de $M$ telle que $\pi$ est une sous-représentation de $\Ind_P^G(\sigma)$. Le couple $(P, \sigma)$ est unique à conjugaison près.
\end{thm}
\begin{proof}
	Voir \cite{Goldberg} théorème 2.18, corollaire 3.2.
\end{proof}

Soit $\sigma \in \Rep(M)$ une représentation irréductible de $M$, alors $\sigma$ est cuspidale si et seulement si $\sigma_{|M^0}$, sa restriction à $M^0$, l'est. D'après un théorème de Harish-Chandra, c'est encore équivalent à dire que $\sigma_{|M^1}$ est compacte. On vérifie alors aisément que le groupe $X(M) = \Hom(M/M^1, \C^*)$ des caractères non ramifiés de $M$ agit par torsion sur les représentations irréductibles cuspidales de $M$. Si $\sigma \in \Rep(M)$ est une représentation de $M$, on peut réaliser tous les $\pi_{\chi} = \Ind_P^G(\chi \sigma)$ dans un même espace vectoriel par découpage en classes modulo $G^0$, puis en restreignant à un bon sous groupe compact maximal de $G^0$ (c'est à dire qu'à $f$ on associe $(f_{|sK_0})_{s \in S}$). Cela permet de définir la notion de régularité (polynomialité, rationnalité, holomorphie, méromorphie par rapport aux paramètres complexes du caractère $\chi$) comme dans le cas connexe\,\footnote{On effectue un découpage car en général il n'existe pas de bon sous-groupe ouvert compact maximal $K_0$ dans $G$ (au sens où rien ne garantit que $G = P K_0$). Si l'on dispose d'un tel $K_0$ (par exemple si $G$ est le produit semi-direct de $G^0$ par un groupe fini), alors on peut directement adapter la méthode précédente. La notion de régularité ne dépend pas de la méthode choisie.}. Plus précisément, si on note $V_{\chi}$ l'espace de fonctions associé à $\Ind_P^G(\chi \sigma)$, et si $f \in V_{\chi}$, alors l'image du morphisme injectif $f \mapsto (f_{|sK_0})_{s \in S}$ est un espace vectoriel qui ne dépend pas de $\chi$. On identifie tous les $V_{\chi}$ à cet unique espace vectoriel pour parler de régularité.

% : On se donne $K_0$ un sous-groupe compact de $G^0$, si $S$ est le système de représentants de $G/G^0$ que l'on s'est fixé , une fonction $f \in V_{\pi_{\chi}}$ s'écrit $f = \sum_{s \in S} f_s * \delta_s$ avec $f_s$ fonction à support dans $G^0$. A la fonction $f$ on associe alors $(f_{s|K_0})_{s \in S}$
% en restreignant à $SK_0$ (où $K_0$ est un sous-groupe compact maximal de $G^0$ et $S$ le système de représentants de $G/G^0$ que l'on s'est fixé).

% subsection représentations_cuspidales (end)

\subsection{Opérateurs d'entrelacements} % (fold)
\label{sub:operateurs_d_entrelacements}

Si $P = M N$, $P' = M N'$ sont des groupes paraboliques cuspidaux de même composante de Levi, $\sigma \in \Rep(M)$ une représentation irréductible cuspidale de $M$ et $\chi \in X(M)$ un caractère non ramifié de $M$, on définit, sous réserve de convergence, un entrelacement $J_{P| P'}(\sigma)$ entre $\pi = \Ind_P^G(\sigma)$ et $\pi' = \Ind_{P'}^G(\sigma)$ par
\[
	J_{P| P'}(\sigma) = \int_{N / N \cap N'} \lambda(n) \, \D n 
\]
où $\lambda(n) : V_{\pi} \to V_{\pi'}$ est la translation de $n$ à gauche.% Sur les fonctions

On peut aussi construire cet entrelacement à partir de ceux définis sur $G^0$ (ce qui permet de déduire facilement leurs propriétés). La restriction $\sigma_{|M^0}$ est semi-simple de longueur finie (puisque $M^0$ est distingué d'indice fini dans $M$). Soit $\sigma_0 \in \Rep(M^0)$ une représentation irréductible qui apparaît comme facteur direct de $\sigma_{|M^0}$. Par réciprocité de Frobenius\,\footnote{On utilise la réciprocité de Frobenius dans les deux sens (l'induction et l'induction compacte de $M^0$ à $M$ coïncident).}, $\sigma$ est un facteur direct de $\Ind_{M^0}^M \sigma_0$. Et par induction, on en déduit finalement que $\pi = \Ind_P^G(\sigma)$ est un facteur direct de $\pi_0 = \Ind_{P^0}^G(\sigma_0)$. De même, $\pi' = \Ind_{P'}^G(\sigma)$ est un facteur direct de $\pi'_0 = \Ind_{P'^0}^G(\sigma_0)$. On a donc des $G$-entrelacements $i : \pi \to \pi_0$ et $s : \pi_0 \to \pi$ (respectivement injectif et surjectif) tels que $s \circ i = \Id_{\pi}$, et la même chose entre $\pi'_0$ et $\pi'$. Du $G^0$-entrelacement

\[
	J_{P^0 | P'^0}(\sigma_0) : \Ind_{P^0}^{G^0}(\sigma_0) \longrightarrow \Ind_{P'^0}^{G^0}(\sigma_0)
\]
on tire par induction un $G$-entrelacement
\[
	\Ind_{G^0}^G \pars{J_{P^0 | P'^0}(\sigma_0)} : \pi_0 \longrightarrow \pi'_0
\]
Et on vérifie avec la formule intégrale que le diagramme suivant commute (voir \cite{Goldberg} lemma 5.5)
\[
	\xymatrix{
		\pi_0 \ar[rrr]^{\Ind_{G^0}^G \pars{J_{P^0 | P'^0}(\sigma_0)}} &&& \pi_0'  \\
		\pi \ar@{^{(}->}[u]^{i} \ar[rrr]_{J_{P | P'}(\sigma)} &&& \pi' \ar@{^{(}->}[u]_{i'}
	}
\]
puis que 
\[
	\xymatrix{
		\pi_0 \ar[rrr]^{\Ind_{G^0}^G \pars{J_{P^0 | P'^0}(\sigma_0)}} &&& \pi_0' \ar@{->>}[d]^{s'} \\
		\pi \ar@{^{(}->}[u]^{i} \ar[rrr]_{J_{P | P'}(\sigma)} &&& \pi'
	}
\]
On peut rétrospectivement utiliser ce dernier diagramme comme définition de $J_{P | P'}(\sigma)$ dès lors que $J_{P^0 | P'^0}(\sigma_0)$ est défini ce qui prolonge sa définition hors du domaine de convergence de l'intégrale.

\section{Représentations d'une groupe réductif non connexe} % (fold)
\label{sec:representations_d_une_groupe_reductif_non_connexe}

\subsection{Le lemme géométrique} % (fold)
\label{sub:le_lemme_geometrique}

On fixe $A_0$ un tore maximal déployé de $G$, et on pose $M_0 = C_G(A_0)$. Comme le tore $A_0$ est maximal, alors $M_0$ est un sous-groupe de Levi cuspidal (puisque $A_0$ est spécial au sens de la définition \ref{def:tore_special}) minimal, et sa composante neutre $M_0^0$ est un sous-groupe de Levi minimal de $G^0$ (d'après la remarque \ref{rem:croissance_composante_connexe}). On fixe $P_0 = M_0 N_0$ un groupe parabolique cuspidal minimal de $G$ (qui se trouve au dessus de $P^0_0$ un sous-groupe parabolique minimal de $G^0$). On dira qu'un sous-groupe parabolique cuspidal $P$ de $G$ est semi-standard si $A_0 \subset P$ (ce qui équivaut à dire que $P^0$ est semi-standard) et standard si en plus on a $P_0 \subset P$ (ce qui équivaut à dire que $P^0$ est standard).

\begin{lemme}
	\label{lemme:nombre_fini_orbites}
	Soient $P, Q \subset G$ des sous-groupes paraboliques cuspidaux de $G$, Alors $Q^0 \setminus G / P^0$ et $Q \setminus G / P$ sont finis.
\end{lemme}
\begin{proof}
	On choisit $S \in G$ un système de représentants de $G/G^0$. Alors $S$ est fini, et on peut écrire
	\[
		G = \bigsqcup_{s \in S} s \, G^0
	\]
	Pour $s \in S$, les groupes $s^{-1} Q^0 s$ et $P^0$ sont des sous-groupes paraboliques de $G^0$, il existe un ensemble fini $W_s \subset G^0$ tel que l'on a
	\[
		G^0 = \bigsqcup_{w \in W_s} (s^{-1} Q^0 s) \, w \, P^0
	\]
	En combinant les deux décompositions, on tire
	\[
		G = \bigsqcup_{s \in S} \bigsqcup_{w \in W_s} Q^0 \, s w \, P^0
	\]
	Donc $Q^0 \setminus G /P^0$ est fini, et a fortiori $Q \setminus G / P$ l'est aussi.
\end{proof}
\begin{lemme}
	\label{lemme:decomposabilite}
	Soient $P = M U$ et $Q = N V$ des sous-groupes paraboliques cuspidaux semi-standards de $G$, alors
	
	\[
		M \cap Q = (M \cap N) . (M \cap V), \quad \quad U \cap Q = (U \cap N) . (U \cap V),  \quad \quad P \cap Q = (P \cap N) . (P \cap V)
	\]
\end{lemme}
\begin{proof}
	Voir \cite{Goldberg} lemme 3.11.
	% Comme $U \subset G^0$, alors on a
	% \[
	% 	U \cap Q = U \cap Q^0 = (U \cap N^0) . (U \cap V) = (U \cap N) . (U \cap V)
	% \]
	% On a que $M \cap Q$ est un sous-groupe de Levi de $M$ (puisque $M / M \cap Q \hookrightarrow G/Q$ est projectif) 
	% % M / M \cap Q \subset G/Q
	% dont la même composante connexe est
	% \[
	% 	(M \cap Q)^0 = M^0 \cap Q^0 = (M^0 \cap N^0) . (M \cap V)
	% \]
	% Ce qui implique qu'il admet une décomposition de Levi de la forme $M \cap Q = N' . (M \cap V)$...
	% A compléter	
\end{proof}
% Condition 4 de BZ page 61 ???
	Soient $P, Q$ des sous-groupes paraboliques cuspidaux. On considère l'espace totalement discontinu $X = P \setminus G$ sur lequel $Q$ agit par translation à droite (c'est-à-dire $\rho(q).(Pg) = (Pg q^{-1})$). On sait d'après le lemme \ref{lemme:nombre_fini_orbites} que $Q$ a un nombre fini d'orbites sur $X$, on choisit une numérotation $Z_i = P w_i Q$ des doubles classes telles que les ensembles
		
	\[
		Y_j =\bigcup_{i \le j} Z_i
	\]
	soient ouverts dans $X$ (ce qui est possible en vertu de \cite{Zelevinskii} 1.5). On note
	\[
		M_i = M \cap w_i^{-1}(N), \quad N_i = w_i(M_i) = w_i(M) \cap N, \quad V_i = M \cap w_i^{-1}(V), \quad U_i = N \cap w_i(U)
	\]
	Et
	\[
		P_i = M_i U_i, \quad Q_i = N_i V_i
	\]
	
	On définit le foncteur $F = r_Q^G \circ i_P^G$. On peut appliquer le lemme géométrique \cite{BZ} théorème 5.2.
\begin{lemme}[Lemme géométrique, \cite{BZ} théorème 5.2]
	Le foncteur $F = r_Q^G \circ i_P^G$ admet une filtration par des sous foncteurs $0 = F_0 \subset F_1 \subset \ldots \subset F_k = F$ tels que $F_i/F_{i-1} \simeq i_{P_i}^{G} \circ w_i \circ r_{Q_i}^{M}$.
\end{lemme}

\begin{proof}
	La preuve de ce résultat est donnée dans \cite{BZ} théorème 5.2 (que l'on peut appliquer en vertu des lemmes \ref{lemme:nombre_fini_orbites} et \ref{lemme:decomposabilite}). Rappelons juste comment sont définis les foncteurs $F_i$. %, et comment on obtient les isomorphismes $F_i/F_{i-1} \simeq i_{P_i}^{G} \circ w_i \circ r_{Q_i}^{M}$.
	Soit $i \in \ints{1,k}$. Soit $(\sigma, V) \in \Rep(M)$ une représentation de $M$. On note $\pi = i_P^G \sigma$, $W = i_P^G V$, et $W_i \subset W$ le sous-espace de $W$ des fonctions à support dans $Y_i$. Ce sous-espace est $Q$-stable, soit $\pi_i$ la sous-représentation de $\pi_{|Q}$ associée. On définit alors $F_i(\sigma) = j_V(\pi_i)$ (où $j_V$ est le foncteur des co-invariants, c'est-à-dire que $r_Q^G(\pi) = j_V(\pi_{|Q})$). Alors les $F_i$ définissent des foncteurs et une filtration de $F$.
	 % 	
	 % 	L'application
	 % 	\begin{align*}
	 % 		W & \longrightarrow F(W) \\
	 % 		f & \longmapsto \int_{V_i \setminus V} j_{V_i}(f(w^{-1} v n)) \, \D \mu(v) 
	 % 	\end{align*}
	 % 	Réalise l'isomorphisme.
\end{proof}

% subsection le_lemme_géométrique (end)

\subsection{Lemme de Jacquet} % (fold)
\label{sub:lemme_de_jacquet}

% Une remarque importante est que les sous-groupes ouverts compacts suffisamment petits sont dans $G^0$. 
Si $K \subset G$ est un sous-groupe ouvert compact, et $H \subset G$ un sous-groupe fermé de $G$, on note $K_H = K \cap H$. On sait d'après le théorème de Bruhat que $G^0$ admet des sous-groupes ouverts compacts $K$ arbitrairement petits admettant une décomposition d'Iwahori selon les sous-groupes paraboliques standard de $G^0$, c'est-à-dire que  pour tous sous-groupe $P^0 = M^0 N$ parabolique standard de $G^0$, on a $K =  K_{\overline{N}} . K_{M^0} . K_N$. Comme $K \subset G^0$, on a $K_M = K_{M^0}$. Pour tout $P = MN$ sous-groupe parabolique cuspidal standard de $G$, on peut donc écrire  $K =  K_{\overline{N}} . K_{M} . K_N$. On dira qu'un tel sous-groupe est en \textbf{de type Iwahori et bonne position} par rapport à $(P,M)$.

\begin{lemme}[Jacquet]
	Soit $(\pi, V) \in \Rep(G)$ une représentation lisse de $G$, $P = MN$ un sous-groupe parabolique cuspidal de $G$, et $K \subset G^0$ un sous-groupe ouvert compact de type Iwahori et en bonne position par rapport à $(P,M)$. Alors l'application
	\[
		p : V^K \to J_N(V)^{K_M}
	\]
	Est surjective. De plus, $p$ possède une section naturelle, ou en d'autres termes, $J_N(V)^{K_M}$ peut-être réalisé comme un facteur direct de $V^K$ d'une manière fonctorielle en $V$.
\end{lemme}
\begin{proof}
	Puisque $K \subset G^0$, les $K$-fixes de $(\pi, V)$ sont ceux de $(\pi_{|G^0}, V)$ et on peut directement invoquer le lemme de Jacquet dans $G^0$ (voir \cite{Bernstein} p.65).
\end{proof}
% subsection lemme_de_jacquet (end)

% section représentations_d_une_groupe_réductif_non_connexe (end)

\section{Le théorème} % (fold)
\label{sec:le_theoreme}

%% Réécrire ?

Passons maintenant à la démonstration du théorème \ref{thm:paley_wiener}. Expliquons succinctement la démarche. On se donne $\varphi$ vérifiant les conditions de l'énoncé. On rappelle que l'on note $\cat$ la sous-catégorie pleine de $\Rep(G)$ stable par somme et sous-quotient engendrée par les représentations $\Ind_{P}^{G} \sigma$ (où $P = M N$ parcourt les sous-groupes paraboliques cuspidaux de $G$, et $\sigma$ les représentations irréductibles cuspidales de $M$), et $\cat^0$ la catégorie correspondante pour $G^0$. Notons $\efo(\cat)$ l'algèbre des endomorphismes du foncteur d'oubli $F_{\cat} : \cat \to \C \textnormal{-Vect}$\,\footnote{C'est-à-dire qu'un élément $\varphi \in \efo(\cat)$ est la donnée pour tout $\pi \in \cat$ d'un endomorphisme $\varphi(\pi) \in \End_{\C}(\pi)$ tel que pour tout $\pi_1, \pi_2 \in \cat$ et $\alpha \in \Hom_{\cat}(\pi_1, \pi_2)$, on a $\alpha \circ \varphi(\pi_1) = \varphi(\pi_2) \circ \alpha$.}, et $\efo^{\infty}(\cat)$ la sous-algèbre des éléments $\varphi$ (dit lisses) tels qu'il existe un sous-groupe ouvert compact $K$ tel que $\widehat{e_K} \varphi = \varphi \widehat{e_K} = \varphi$. On adopte des notations identiques pour $G^0$.% La commutation aux entrelacements (condition 3) permet d'étendre la définition de $\varphi$ à $\cat$ toute entière, c'est-à-dire $\varphi \in \efo(\cat)$. La condition 2 assure par surcroît que $\varphi \in \efo(\cat)$.

On cherche $f \in \Hecke(G)$ tel que $\varphi = \widehat{f}$. Notre démarche va consister à construire (sans avoir recours à $f$) un élément $\varphi_{|G^0} \in \efo(G^0)$ tel que si $\varphi = \widehat{f}$ alors $\varphi_{|G^0} = \widehat{f_{|G^0}}$. Il s'agira alors montrer que l'objet $\varphi_{|G^0}$ ainsi construit est redevable des hypothèses de la version connexe du théorème de Paley-Wiener (théorème \ref{thm:paley_wiener_connexe}). Enfin, pour terminer la démonstration, on applique cette démarche aux translatés de $\varphi$ permettant de recouvrer $f_{|gG^0}$ pour tout $g \in G^0$.

\subsection{Construction de $\varphi(\Ind_{P^0}^G(\sigma_0))$} % (fold)
\label{sub:construction_de_varphi_p^0_g_sigma_0_}

%Rappelons que la commutation aux entrelacements permet d'étendre naturellement $\varphi$ aux sommes directes de représentations de la forme $\Ind_P^G(\sigma)$ (ainsi qu'aux quotients). C'est donc en particulier possible sur les représentations de la forme $\Ind_{P^0}^G(\sigma_0)$ où $P^0 = M^0 N$ est un sous-groupe parabolique de $G^0$ et $\sigma_0 \in \Rep(M^0)$ une représentation irréductible cuspidale de $M^0$. En effet, si $\sigma_0$ est une telle représentation, alors $\Ind_{M^0}^M(\sigma_0)$ est semi-simple et de longueur finie. On se donne une décomposition en sous-espaces irréductibles
On se donne $P^0 = M^0 N$ est un sous-groupe parabolique de $G^0$, et $\sigma_0 \in \Rep(M^0)$ une représentation irréductible cuspidale de $M^0$, on cherche à définir $\varphi(\Ind_{P^0}^G(\sigma_0'))$ pour tout $\sigma_0'$ dans l'orbite de $\sigma_0$. Pour cela, on se donne une décomposition en sous-espaces irréductibles
\[
	\Ind_{M^0}^M(\sigma_0) = \bigoplus_{i = 1}^n \sigma_i
\]
Notons qu'il s'agit ici d'un choix non canonique, et que l'on ne suppose pas que les $\sigma_i$ sont non isomorphes entre eux. En induisant, on a donc (en identifiant $\Ind_{P^0}^G(\sigma_0)$ et $\Ind_{P}^G \Ind_{M^0}^{M}(\sigma_0)$)
\[
	\Ind_{P^0}^G(\sigma_0) = \bigoplus_{i = 1}^n \Ind_{P}^G (\sigma_i)
\]
Or d'après \cite{Goldberg} lemme 2.15, les représentations $\sigma_i \in \Rep(M)$ qui apparaissent sont toutes cuspidales. Ce qui permet d'étendre $\varphi$ par la formule\,\footnote{Pour être rigoureux il faudrait prendre en compte l'identification entre $\Ind_{P^0}^G(\sigma_0)$ et $\Ind_{P}^G \Ind_{M}^{M^0}(\sigma_0)$ ici aussi. Nous l'avons omis pour alléger la formule.}
\[
	\varphi(\Ind_{P^0}^G(\sigma_0)) = \bigoplus_{i = 1}^n \varphi(\Ind_{P}^G (\sigma_i))
\]
On vérifie que comme les $\varphi(\Ind_P^G \sigma_i)$ commutent aux entrelacements $\lambda(g)$ et $J_{P|P'}(\sigma_i)$, alors les $\varphi(\Ind_{P^0}^G(\sigma_0))$ commutent aux entrelacements $\lambda(g)$ et $\Ind_{G^0}^G(J_{P^0, P^{0'}}(\sigma_0))$ (parce qu'on a la décomposition $\Ind_{G^0}^G(J_{P^0, P^{0'}}(\sigma_0)) = \bigoplus_{i= 1}^n J_{P|P'}(\sigma_i)$, et une décomposition similaire pour l'opérateur $\lambda$).
% En vertu de la condition de commutation aux entrelacements, $\varphi(\Ind_{P^0}^G(\sigma_0))$ ne dépend en fait pas du choix de la décomposition de $\Ind_{M^0}^M(\sigma_0)$ effectué précédemment. Par construction, cette extension de $\varphi$ commute encore aux $G$-entrelacements et conserve la même lissité.

Par ailleurs, si l'on tord $\sigma_0$ par un caractère non ramifié $\chi$ de $M^0$, on trouve la même décomposition où les $\sigma_i$ sont tordues par le même caractère (prolongé à $M$). On prolonge donc notre définition sur toute l'orbite par
\[
	\varphi(\Ind_{P^0}^G(\chi \otimes \sigma_0)) = \bigoplus_{i = 1}^n \varphi(\Ind_{P}^G (\tilde{\chi} \otimes \sigma_i))
\]

Il apparaît que la fonction $\chi \longmapsto \varphi(\Ind_{P^0}^G(\chi \otimes \sigma_0))$ est polynomiale.
%%%% Preuve
%Il est aisé de vérifier que si $\varphi$ coïncide avec $\widehat{f}$ sur les représentations $\Ind_{P}^G (\sigma_i)$ si et seulement si c'est le cas pour $\Ind_{P^0}^G(\sigma_0)$ (le sens direct consiste à vérifier que notre construction donne ce qu'il faut pour $\varphi(\Ind_{P^0}^G(\sigma_0))$, et la réciproque découle de la commutation aux entrelacements puisque les $\Ind_{P}^G (\sigma_i)$ sont des sous-représentations de $\Ind_{P^0}^G(\sigma_0)$).
%%%% Preuve

% subsection construction_de_varphi_p^0_g_sigma_0_ (end)

\subsection{Restriction à $G^0$} % (fold)
\label{sub:restriction_a_g_0_}

On définit maintenant un objet $\varphi_{|G^0}$ sur les représentations de $G^0$ en se basant sur la définition \ref{defn:restriction_a_H}. L'idée de la définition est que si $\varphi = \widehat{f}$ pour $f \in \Hecke(G)$, alors on a $\varphi_{|G^0} = \widehat{f_{|G^0}}$ (cf lemme \ref{lemme:restriction_commute_chapeau}). Il suffira alors de montrer que $\varphi_{|G^0}$ est redevable des hypothèses du théorème \ref{thm:paley_wiener_connexe} dans le cas connexe.
On se donne $S$ un système de représentants de $G/G^0$. La représentation $\Ind_{P^0}^{G^0}(\sigma_0^s)$ est un facteur direct de $\Ind_{P^0}^G(\sigma_0)_{|G^0}$. Une $G^0$-surjection est fournie par la restriction à $G^0$, définie par

\begin{align*}
	e_s : \Ind_{P^0}^G(\sigma_0) &\longrightarrow \pars{\Ind_{P^0}^{G^0}\sigma_0}^s\\
	f &\longmapsto (g_0 \mapsto f(g_0 s))
\end{align*}
tandis qu'une $G^0$-injection est fournie par l'extension par $0$, définie par
\begin{align*}
	e_s^* : \pars{\Ind_{P^0}^{G^0}\sigma_0}^s &\longrightarrow \Ind_{P^0}^{G}(\sigma_0)\\
	f &\longmapsto \left(g \mapsto \begin{cases}
		f(g s^{-1}) & \textrm{si }g \in sG^0 \\
		0 & \textrm{sinon}
	\end{cases}\right)
\end{align*}
On a $e_s \circ e_s^* = \Id_{\pars{\Ind_{P^0}^{G^0}\sigma_0}^s}$ et $p_s = e_s^* \circ e_s$ est un $G^0$-projecteur de $\Ind_{P^0}^{G}(\sigma_0)$ (la projection sur les fonctions à support dans $sG^0$). On définit alors $\varphi_{|G^0}$ par

\[
	\varphi_{|G^0}(\Ind_{P^0}^{G^0} \, \chi\sigma_0) = e_1 \circ \varphi(\Ind_{P^0}^{G}(\chi\sigma_0)) \circ e_1^*
\]
Il est facile de vérifier que $\varphi_{|G^0}$ satisfait les conditions de lissité et de polynomialité. Il reste à vérifier les conditions de commutations aux entrelacements nécessaires à l'application de \cite{Heiermann} théorème 0.1. La commutation aux translations à gauche vient du fait que si $g_0 \in G^0$ l'opérateur $\lambda(g_0) : \Ind_{P^0}^{G}(\sigma_0) \to \Ind_{g \cdot P^0}^{G}(g \cdot \sigma_0)$ induit par restriction la translation à gauche de $\Ind_{P^0}^{G^0}(\sigma_0)$ vers $\Ind_{g \cdot P^0}^{G}(g \cdot \sigma_0)$ et que les $\varphi(\Ind_{P^0}^G(\sigma_0))$ commutent aux translations à gauche. De même, la vérification de la commutation aux entrelacements $J_{P^0, P^{0'}}(\sigma_0)$ se fait en utilisant une propriété similaire des opérateurs $\Ind_{G^0}^G(J_{P^0, P^{0'}}(\sigma_0))$.
% On vérifie alors que $\varphi_{|G^0} \in \efo^{\infty}(\cat^0)$ (voir lemme \ref{lemme:restriction_commute}),
%%%% Preuve!!
% et reste polynomiale.
%%%% Preuve!!
D'après \cite{Heiermann} théorème 0.1, il existe donc $f_1 \in \Hecke(G^0)$ telle que $\varphi_{|G^0} = \widehat{f_1}$.

L'idée pour conclure est maintenant d'appliquer ce raisonnement à $(\widehat{\delta}_{s^{-1}} \varphi)_{s \in S}$ (où $\widehat{\delta}_{g} \in \efo(\cat)$ désigne l'élément $\pi \mapsto \pi(g)$). Pour chaque $s \in S$ on trouve une fonction $f_s \in \Hecke(G^0)$ telle que $(\widehat{\delta}_{s^{-1}} \varphi)_{|G^0} = \widehat{f_s}$. On définit alors $f \in \Hecke(G)$ par

\[
	f = \sum_{s \in S} \delta_s * f_s
\]
Un calcul (voir lemme \ref{lemme:decomposition}) montre que $\widehat{f}$ et $\varphi$ coïncident sur les représentations du type $\Ind_{P^0}^{G}(\sigma_0)$, puis par identification des termes diagonaux, $\widehat{f}$ et $\varphi$ coïncident sur les représentations du type $\Ind_{P}^{G}(\sigma_0)$.  
%%%% Faire le calcul
% 
% \begin{align*}
% 	\varphi &= \sum_s \sum_t p_s \varphi p_t \\
% 	&= \sum_s \sum_t p_s \varphi \Pi(s^{-1}t) p_s \Pi(s^{-1}t) \\
% 	&= \sum_s \sum_{t'} p_s \varphi \Pi(t') p_s \Pi(t'^{-1}) \\
% 	&= \sum_{t'} \varphi_{t'} \Pi(t'^{-1})
% \end{align*}
Ce qui prouve le théorème \ref{thm:paley_wiener}.
%%%% Etendre

\subsection{Une relation polynomiale} % (fold)
\label{sub:une_relation_polynomiale}

La preuve du théorème 0.1 de \cite{Heiermann} fournit en outre une formule d'inversion (en fait, la preuve du théorème repose sur cette formule). Nous allons maintenant donner une formule analogue dans le cas non connexe. Pour cela, on va démontrer un résultat analogue à la proposition 0.2 de \cite{Heiermann}. En gardant les notations du paragraphe précédent, on fixe $s \in S$ et on pose
\[
	\phi_s(\Ind_{P^0}^{G^0} \, \sigma_0) = e_1 \circ \Ind_{P^0}^{G}(\sigma_0)(s^{-1}) \circ \varphi(\Ind_{P^0}^{G} \, \sigma_0) \circ e_1^* = e_{s^{-1}} \circ \varphi(\Ind_{P^0}^{G}(\sigma_0)) \circ e_1^*
\]
Alors d'après \cite{Heiermann} proposition 0.2, il existe une application $\xi_{s}(\Ind_{P^0}^{G^0}\sigma_0) : \Ind_{P^0}^{G^0} \sigma_0 \to \Ind_{\overline{P^0}}^{G^0} \sigma_0$, polynomiale en $\sigma_0$ telle que pour tout $\sigma_0 \in \Orb_0$
\[
	\phi_s(\Ind_{P^0}^{G^0} \, \sigma_0) = \sum_{w \in W(M^0, \Orb_{0})}
	 J_{P^0 | w \overline{P^0}}(\sigma_0) \circ \lambda(w) \circ
	\xi_{s}(\Ind_{P^0}^{G^0}w^{-1}\sigma_0)
	\circ \lambda(w)^{-1} \circ J_{wP^0 | P^0}(\sigma_0)
\]
Où on a noté $W(M^0, \Orb_{0}) = \braces{w \in W(M^0), \ w \Orb_{0} = \Orb_{0}}$. On note $\zeta(\Ind_{P^0}^{G^0}\, \sigma_0) = J_{\overline{P^0} | P^0}(\sigma_0)^{-1} \ \xi(\Ind_{P^0}^{G^0} \, \sigma_0)$. La fonction $\sigma_0 \mapsto \zeta(\Ind_{P^0}^{G^0}\, \sigma_0)$ est rationnelle (ses pôles sont les points non-inversibilité de $J_{\overline{P^0} | P^0}(\sigma_0)$).

\subsubsection{Sur les représentations du type $\Ind_{P^0}^G \, \sigma_0$} % (fold)
\label{ssub:sur_les_representations_du_type_p^0_g_,_sigma_0_}
On en déduit un résultat analogue pour les représentations du type $\Ind_{P^0}^G(\sigma_0)$.

\begin{lemme}
	\label{lemme:0.2s_Ind_P0}
	Il existe une application polynomiale $\xi_{s}(\Ind_{P^0}^{G}\sigma_0) : \Ind_{P^0}^{G} \, \sigma_0 \to \Ind_{\overline{P^0}}^{G} \, \sigma_0$, polynomiale en $\sigma_0$ telle que pour tout $\sigma_0 \in \Orb_0$, on ait
	\[
		\varphi_s(\Ind_{P^0}^{G} \, \sigma_0) = \sum_{w \in W(M^0, \Orb_{0})}
		 J_{P^0 | w \overline{P^0}}^G(\sigma_0) \circ \lambda(w) \circ
		\xi_{s}(\Ind_{P^0}^{G}w^{-1}\sigma_0)
		\circ \lambda(w)^{-1} \circ J_{wP^0 | P^0}^G(\sigma_0)
	\]
\end{lemme}
\begin{proof}
	On a un $G^0$-isomorphisme
	\begin{equation}
		\label{eq:decomp}
		(\Ind_{P^0}^{G} \sigma_0)_{|G^0} \underset{\alpha_{P^0}}{\simeq} \bigoplus_{t \in S} \Ind_{t \cdot P^0}^{G^0} \, t \cdot \sigma_0
	\end{equation}
	donné par
	\begin{align*}
		\alpha_{P^0} : \Ind_{P^0}^{G} V_{\sigma_0} & \longrightarrow \bigoplus_{t \in S} \Ind_{t \cdot P^0}^{G^0} \, tV_{\sigma_0} \\
		f & \longmapsto \bigoplus_{t \in S} f_t
	\end{align*}
	Où $f_t$ est définie pour $g_0 \in G^0$ par $f_t(g_0) = f(t^{-1} g_0)$, et $\alpha_{P^0}^{-1}(\Ind_{t \cdot P^0}^{G^0} \, tV_{\sigma_0})$ correspond au sous-espace de $\Ind_{P^0}^{G} V_{\sigma_0}$ des fonctions supportées dans $tG^0$.	L'isomorphisme \ref{eq:decomp} est également valable pour $(\Ind_{\overline{P^0}}^{G} \sigma_0)_{|G^0}$ (en changeant $P^0$ en $\overline{P^0}$).
	%$f_t(g_0) = f(t^{-1} g_0)$ pour $g_0 \in G^0$)
	On définit $\xi_s(\Ind_{P^0}^{G} \sigma_0) : \Ind_{P^0}^{G} \sigma_0 \to \Ind_{\overline{P^0}}^{G} \sigma_0$ diagonalement vis-a-vis de la décomposition (\ref{eq:decomp}) %(cf. annexe \ref{sub:somme_directe})
	par la formule
	\begin{equation}
		\xi_s(\Ind_{P^0}^{G} \sigma_0) = \alpha_{\overline{P^0}}^{-1} \circ \pars{\bigoplus_{t \in S} \xi_s (\Ind_{t \cdot P^0}^{G^0} \, t \cdot \sigma_0)} \circ \alpha_{P^0}
	\end{equation}
	% C'est-à-dire que l'on a
	% 	\[
	% 		\xi_s(\Ind_{P^0}^{G} \sigma_0)(f) = \alpha_{\overline{P^0}}^{-1} \left[ \pars{\xi_s (\Ind_{t \cdot P^0}^{G^0} \, t \cdot \sigma_0)(f_t)}_{t \in S} \right]
	% 	\]
	% Où la décomposition $(\Ind_{P^0}^{G} \sigma_0)_{|G^0} = \bigoplus_{t \in S} \Ind_{t \cdot P^0}^{G^0} \, t \cdot \sigma_0$ est donnée par $f \mapsto (f_{t})_{t \in S}$ (avec $f_t(g_0) = f(t^{-1} g_0)$ pour $g_0 \in G^0$) et la somme directe de morphismes est prise au sens donné en \ref{sub:somme_directe}.
	Il est clair par construction que $\sigma_0 \mapsto \xi_s(\Ind_{P^0}^{G} \sigma_0)$ est polynomiale (puisqu'elle l'est sur chaque bloc de la diagonale). Il reste à vérifier la relation annoncée. Comme $\varphi_s$ est à support dans $G^0$, il agit diagonalement (toujours par rapport à la décomposition \ref{eq:decomp}), c'est-à-dire précisément que l'on a
	\begin{align*}
		\alpha_{\overline{P^0}} \circ \varphi_s(\Ind_{P^0}^{G} \, \sigma_0)  \circ \alpha_{P^0}^{-1} &= \bigoplus_{t \in S} \phi_s \pars{\Ind_{t \cdot P^0}^{G^0} \, t \cdot \sigma_0}\\
		&= \bigoplus_{t \in S} \sum_{w_t \in W(t \cdot M^0, t \cdot \Orb_{0})}
		 J_{t \cdot P^0 | w_t t \cdot \overline{P^0}}(t \cdot \sigma_0) \circ \lambda(w_t) \circ
		\xi_{s} \pars{\Ind_{t \cdot P^0}^{G^0} w_t^{-1} \, t \cdot \sigma_0}
		\circ \lambda(w_t)^{-1} \circ J_{w_t t \cdot P^0 | t \cdot P^0}(t \cdot \sigma_0)
	\end{align*}
	Or $w_t \in W(t \cdot M^0, t \cdot \Orb_{0})$ si et seulement si $t^{-1} w_t t \in W(M^0, \Orb_{0})$. Par changement de variable $w = t^{-1} w_t t$, on tire donc
	\begin{align*}
		\alpha_{\overline{P^0}} \circ \varphi_s(\Ind_{P^0}^{G} \, \sigma_0) \circ \alpha_{P^0}^{-1} &= \bigoplus_{t \in S} \sum_{w \in W(M^0, \Orb_{\sigma_0})}
		 J_{t \cdot P^0 | t w\overline{P^0}}(t \cdot \sigma_0) \circ \lambda(t w t^{-1}) \circ
		\xi_{s} \pars{\Ind_{t \cdot P^0}^{G^0} \, t w^{-1} \sigma_0}
		\circ \lambda(t w^{-1} t^{-1}) \circ J_{t w P^0 | t \cdot P^0}(t \cdot \sigma_0) \\
		&= \sum_{w \in W(M^0, \Orb_{\sigma_0})}  \bigoplus_{t \in S}
		 J_{t \cdot P^0 | t w\overline{P^0}}(t \cdot \sigma_0) \circ \lambda(t w t^{-1}) \circ
		\xi_{s} \pars{\Ind_{t \cdot P^0}^{G^0} \, t w^{-1} \sigma_0}
		\circ \lambda(t w^{-1} t^{-1}) \circ J_{t w P^0 | t \cdot P^0}(t \cdot \sigma_0)
	\end{align*}
	Et grâce aux relations (toujours selon la décomposition  \ref{eq:decomp})
	\[
		\alpha_{Q^0} \circ J_{P^0 | Q^0}^G(\sigma_0) \circ \alpha_{P^0}^{-1} = \bigoplus_{t \in S} J_{t \cdot P^0 | t Q^0}(t \cdot \sigma_0) \quad \quad \quad \text{ et } \quad \quad \quad \alpha_{\overline{P^0}} \circ \lambda(w) \circ \alpha_{P^0}^{-1} = \bigoplus_{t \in S} \lambda(t w t^{-1})
	\]
	on conclut que
	\[
		\varphi_s(\Ind_{P^0}^{G} \, \sigma_0) = \sum_{w \in W(M^0, \Orb_{\sigma_0})} J_{P^0 | w \overline{P^0}}^G(\sigma_0) \circ \lambda(w) \circ
		\xi_{s}(\Ind_{P^0}^{G}w^{-1}\sigma_0)
		\circ \lambda(w)^{-1} \circ J_{wP^0 | P^0}^G(\sigma_0)
	\]
\end{proof}

Etant donnée une application $\xi$ comme dans le lemme \ref{lemme:0.2s_Ind_P0}, on posera $\zeta_s(\Ind_{P^0}^{G}\, \sigma_0) = J_{\overline{P^0} | P^0}^G(\sigma_0)^{-1} \ \xi_s(\Ind_{P^0}^{G} \, \sigma_0)$. La fonction $\sigma_0 \mapsto \zeta(\Ind_{P^0}^{G}\, \sigma_0)$ est rationnelle (ses pôles sont les zéros de $J_{\overline{P^0} | P^0}^G(\sigma_0)$). On a alors la relation
\[
	\zeta_s(\Ind_{P^0}^{G}\, \sigma_0) = \bigoplus_{t \in S} \zeta_s \pars{\Ind_{t \cdot P^0}^{G^0} \, t \cdot \sigma_0}
\]
En assemblant les morceaux de lemme précédent pour $s \in S$, on déduit le lemme suivant.

\begin{lemme}
	\label{lemme:0.2_Ind_P0}
	Il existe une application $\xi(\Ind_{P^0}^{G}\sigma_0) : \Ind_{P^0}^{G} \, \sigma_0 \to \Ind_{\overline{P^0}}^{G} \, \sigma_0$, polynomiale en $\sigma_0$ telle que pour tout $\sigma_0 \in \Orb_0$, on ait
	\[
		\varphi(\Ind_{P^0}^{G} \, \sigma_0) = \sum_{w \in W(M^0, \Orb_{0})}
		 J_{P^0 | w \overline{P^0}}^G(\sigma_0) \circ \lambda(w) \circ
		\xi(\Ind_{P^0}^{G}w^{-1}\sigma_0)
		\circ \lambda(w)^{-1} \circ J_{wP^0 | P^0}^G(\sigma_0)
	\]
\end{lemme}
\begin{proof}
	Il suffit de poser
	\begin{equation}
		\xi(\Ind_{P^0}^{G}\sigma_0) = \sum_{s \in S} \Ind_{\overline{P^0}}^{G} \, \sigma_0(s) \circ \xi_s(\Ind_{P^0}^{G}\sigma_0)
	\end{equation}
	Il est clair que l'application ainsi définie est polynomiale en $\sigma_0$. Par ailleurs, d'après le lemme \ref{lemme:decomposition}, on a
	\begin{align*}
		\varphi \pars{\Ind_{P^0}^G \sigma_0} &= \sum_{s \in S} \Ind_{P^0}^{G} \, \sigma_0(s) \circ \varphi_s(\Ind_{P^0}^{G}\sigma_0) \\
		&= \sum_{s \in S} \sum_{w \in W(M^0, \Orb_{0})} \Ind_{P^0}^{G} \, \sigma_0(s) \circ J_{P^0 | w \overline{P^0}}^G(\sigma_0) \circ \lambda(w) \circ
		\xi_s(\Ind_{P^0}^{G}w^{-1}\sigma_0)
		\circ \lambda(w)^{-1} \circ J_{wP^0 | P^0}^G(\sigma_0)
	\end{align*}
	Or, comme les $\lambda(w)$ et $J_{P^0 | Q^0}^G(\sigma_0)$ sont des $G$-entrelacements, on tire
	\begin{align*}
		\varphi \pars{\Ind_{P^0}^G \sigma_0} &= \sum_{s \in S} \sum_{w \in W(M^0, \Orb_{0})} J_{P^0 | w \overline{P^0}}^G(\sigma_0) \circ \lambda(w) \circ
		\Ind_{\overline{P^0}}^{G} \, (w^{-1}\sigma_0)(s) \circ \xi_s(\Ind_{P^0}^{G}w^{-1}\sigma_0)
		\circ \lambda(w)^{-1} \circ J_{wP^0 | P^0}^G(\sigma_0) \\
		&= \sum_{w \in W(M^0, \Orb_{0})}
		 J_{P^0 | w \overline{P^0}}^G(\sigma_0) \circ \lambda(w) \circ
		\xi(\Ind_{P^0}^{G}w^{-1}\sigma_0)
		\circ \lambda(w)^{-1} \circ J_{wP^0 | P^0}^G(\sigma_0)
	\end{align*}
\end{proof}

Etant donnée une application $\xi$ comme dans le lemme \ref{lemme:0.2_Ind_P0}, on posera $\zeta(\Ind_{P^0}^{G}\, \sigma_0) = J_{\overline{P^0} | P^0}^G(\sigma_0)^{-1} \ \xi(\Ind_{P^0}^{G} \, \sigma_0)$. La fonction $\sigma \mapsto \zeta(\Ind_{P}^{G}\, \sigma)$ est rationnelle (ses pôles sont les zéros de $J_{\overline{P} | P}(\sigma)$). On a alors la relation
\[
	\zeta(\Ind_{P^0}^{G}\, \sigma_0) = \sum_{s \in S} \Ind_{P^0}^{G} \, \sigma_0(s) \circ \zeta_s(\Ind_{P^0}^{G}\sigma_0)
\]

% subsubsection sur_les_représentations_du_type_p^0_g_,_sigma_0_ (end)

\subsubsection{Sur les représentations du type $\Ind_P^G \, \sigma$} % (fold)
\label{ssub:sur_les_representations_du_type_ind_p_g_,_sigma_}

\begin{prop}
	\label{prop:0.2_Ind_P}
	Il existe une application $\xi(\Ind_{P}^{G}\sigma) : \Ind_{P}^{G} \sigma \to \Ind_{\overline{P}}^{G} \sigma$, polynomiale en $\sigma$ telle que pour tout $\sigma \in \Orb$, on ait
	\[
		\varphi(\Ind_{P}^{G} \, \sigma) = \sum_{w \in W(M, \Orb)}
		 J_{P | w \overline{P}}(\sigma) \circ \lambda(w) \circ
		\xi(\Ind_{P}^{G}w^{-1}\sigma)
		\circ \lambda(w)^{-1} \circ J_{wP | P}(\sigma)
	\]
	où $W(M, \Orb) = \braces{w \in W(M^0), \ w \Orb = \Orb}$.
\end{prop}
\begin{proof}
	Pour tout $\sigma \in \Orb$, on choisit $\sigma_0 \in \Rep(M^0)$ un facteur irréductible de $\sigma_{|M^0}$ de manière $X(M)$-équivariante, c'est-à-dire tel que pour tout $\chi \in X(M)$ et $\sigma \in \Orb$, on a $(\chi\sigma)_0 = \chi_0 \sigma_0$ (en notant $\chi_0 = \chi_{|M^0}$) : il suffit de choisir $\sigma_0$ pour un point fixé de $\sigma \in \Orb$, et comme l'espace $\Hom_{M^0}(\chi_{0} \sigma_0, (\chi\sigma)_{|M^0})$ ne dépend pas de $\chi$, un choix pour tous les autres points en découle. On note $\Orb_0$ l'orbite formée par les $\sigma_0$.
	
	Par réciprocité de Frobenius, $\sigma$ est un facteur direct de $\Ind_{M^0}^M \sigma_0$. On se donne une $M$-injection $i_{\sigma} : \sigma \to \Ind_{M^0}^M \sigma_{0}$ et $M^0$-surjection $s_{\sigma} : \Ind_{M^0}^M \sigma_{0} \to \sigma$ telles que $s_{\sigma} \circ i_{\sigma} = \Id_{\sigma}$. Alors par induction (qui est exacte) on tire une $G$-injection $\Ind_P^G \, i_{\sigma} :  \Ind_P^G \, \sigma \to \Ind_{P^0}^G \, \sigma_{0}$ et une $G$-surjection $\Ind_P^G \, s_{\sigma} : \Ind_{P^0}^G \, \sigma_{0} \to \Ind_{P}^G \, \sigma$. On pose alors
	\begin{equation}
		\xi(\Ind_P^G \, \sigma) = \Ind_{\overline{P}}^G \, s_{\sigma} \circ \xi(\Ind_{P^0}^{G} \, \sigma_0) \circ \Ind_P^G \, i_{\sigma}
	\end{equation}
	Il est clair que la formule définit une application polynomiale en $\sigma$. Par commutation de $\varphi$ aux $G$-entrelacements, on a
	\begin{align*}
		\varphi(\Ind_P^G \, \sigma) &= \Ind_{P}^G \, s_{\sigma} \circ \varphi(\Ind_{P^0}^{G^0} \, \sigma_0) \circ \Ind_P^G \, i_{\sigma}
	\end{align*}
	D'après la section \ref{sub:operateurs_d_entrelacements}, on a les relations
	\[
		\Ind_{Q}^G \, i_{\sigma} \circ J_{P|Q}(\sigma) = J_{P^0|Q^0}^G(\sigma_0) \circ \Ind_P^G \, i_{\sigma} 	\et	  J_{P|Q}(\sigma) \circ \Ind_{P}^G \, s_{\sigma} = \Ind_Q^G \, s_{\sigma} \circ J_{P^0|Q^0}^G(\sigma_0)
	\]
	et comme $\lambda(w) \in \End(\Ind_P^G)$, on vérifie que
	\[
		\lambda(w) \circ \Ind_P^G \, s_{\sigma} = \Ind_{wP}^G \, s_{w\sigma} \circ \lambda(w) 	\et	 \lambda(w) \circ \Ind_P^G \, i_{\sigma} = \Ind_{wP}^G \, i_{w\sigma}
	\]
	En combinant, on tire la relation
	\[
		\varphi_s(\Ind_{P}^{G} \, \sigma) = \sum_{w \in W(M^0, \Orb_0)}
		 J_{P | w \overline{P}}(\sigma) \circ \lambda(w) \circ
		\xi_s(\Ind_{P}^{G}w^{-1}\sigma)
		\circ \lambda(w)^{-1} \circ J_{wP | P}(\sigma)
	\]
\end{proof}
\begin{rem}
	Si on a la décomposition $\pars{\Ind_P^G \, \sigma}_{|G^0} = \bigoplus_{i \in I} \pi_i$, alors les $\pi_i$ sont de la forme $\Ind_{g_i P^0}^{G^0} \, g_i \sigma_0$ pour $g_i \in G$ (en effet, comme $\Ind_P^G \, \sigma \hookrightarrow \Ind_{P^0}^G \, \sigma_0$, on a $\pars{\Ind_P^G \, \sigma}_{|G^0} \hookrightarrow \bigoplus_{g \in G/G^0} \Ind_{g P^0}^{G^0} \, g \sigma_0$). Et poser
	\[
		\xi(\Ind_P^G \, \sigma) = \sum_{s \in S} \Ind_{\overline{P}}^G \, \sigma (s) \bigoplus_{i \in I} \xi_s(\pi_i)
	\]
	Conviendrait.
\end{rem}

Etant donnée une application $\xi$ comme dans le lemme \ref{prop:0.2_Ind_P}, on posera $\zeta(\Ind_{P}^{G}\, \sigma) = J_{\overline{P} | P}^G(\sigma)^{-1} \ \xi(\Ind_{P}^{G} \, \sigma)$. On a alors la relation
\[
	\zeta(\Ind_{P}^{G}\, \sigma) = \Ind_{P}^G \, s_{\sigma} \circ \zeta(\Ind_{P^0}^{G} \, \sigma_0) \circ \Ind_P^G \, i_{\sigma}
\]

% subsubsection sur_les_représentations_du_type_ind_p_g_,_sigma_ (end)

% subsection une_relation_polynomiale (end)

\subsection{Une formule d'inversion} % (fold)
\label{sub:une_formule_d_inversion}
% Donnons une formule d'inversion.

Dans cette partie on donne une formule d'inversion pour la transformée de Fourier sur $G$. On conserve les notations du paragraphe précédent, et on suppose que l'on dispose d'objets $\varphi$, $\xi$ et $\zeta$ définis comme précédemment. On sait qu'il existe $f \in \Hecke(G)$ telle que $\varphi = \widehat{f}$, nous allons donner une formule pour $f$.

\subsubsection{Notations} % (fold)
\label{ssub:notations}

On note $\Theta_0$ l'ensemble des couples $(P^0, \Orb_0)$ où $P^0 = M^0 N$ est un sous-groupe parabolique de $G^0$ et $\Orb_0 = X(M^0). \sigma_0 = \braces{\chi \sigma_0, \chi \in X(M^0)}$ est l'orbite par torsion sous $X(M^0)$ d'une classe d'équivalence de représentation $\sigma_0 \in \Rep(M^0)$ irréductible cuspidale de $M^0$. On note $X(\sigma_0) = \braces{\chi \in X(M^0), \, \chi \sigma_0 \simeq \sigma_0}$ le stabilisateur de $\sigma_0$. C'est un groupe fini (voir \cite{Bernstein} lemme 21) qui ne dépend en fait que de l'orbite $\Orb_0$ (puisque $X(M^0)$ est commutatif). Le groupe $G$ agit sur $\Theta_0$ par conjugaison. On note $\Theta_0/G$ l'ensemble des orbites. Pour $(P^0, \Orb_0) \in \Theta_0$, on note $[P^0, \Orb_0]$ sa classe de $G$-conjugaison et $[\Orb_0]$ la classe de $G$-conjugaison de $\Orb_0$.

De même, on note $\Theta$ l'ensemble des couples $(P, \Orb)$ où $P = M N$ est un sous-groupe parabolique cuspidal de $G$ et $\Orb = X(M) . \sigma = \braces{\chi \sigma, \chi \in X(M)}$ est l'orbite par torsion sous $X(M)$ d'une classe d'équivalence de représentation $\sigma \in \Rep(M)$ irréductible cuspidale de $M$. Le groupe $G$ agit sur $\Theta$ par conjugaison et on note $\Theta/G$ l'ensemble des orbites. Pour $(P, \Orb) \in \Theta$, on note $[P, \Orb]$ sa classe de $G$-conjugaison et $[\Orb]$ la classe de $G$-conjugaison de $\Orb$.

Si $\sigma_0 \in \Rep(M^0)$ est une représentation irréductible cuspidale de $M^0$, alors il existe $\sigma \in \Rep(M)$ une représentation irréductible cuspidale de $M$ telle que $\sigma \hookrightarrow \Ind_{M^0}^M \, \sigma_0$. La représentation $\sigma$ est unique à $X_0(M)$-torsion près. Et on peut écrire
\[
	\Ind_{M^0}^M \, \sigma_0 \simeq m(\sigma_0) \bigoplus_{\psi \in X_0(M)/X_0(\sigma)} \psi \sigma
\]
Par ailleurs, pour tout $\chi \in X(M)$, en notant $\chi_0 = \chi_{|M^0}$, on a
\begin{align*}
	\Hom_{M}( \chi \sigma, \Ind_{M^0}^M \, \chi_0 \sigma_0) &\simeq \Hom_{M^0}((\chi \sigma)_{|M^0}, \chi_0 \sigma_0)\\
	& =  \Hom_{M^0}( \chi_0\sigma_{|M^0}, \chi_0\sigma_0)\\
	& =  \Hom_{M^0}( \sigma_{|M^0}, \sigma_0)\\
	& \simeq \Hom_{M}( \sigma, \Ind_{M^0}^M \, \sigma_0)
\end{align*}
Donc $\chi \sigma \hookrightarrow \Ind_{M^0}^M \, \chi_0 \sigma_0$. Et $m(\chi \sigma) = \dim \Hom_{M}( \chi \sigma, \Ind_{M^0}^M \, \chi_0 \sigma_0) = \dim \Hom_{M}(\sigma, \Ind_{M^0}^M \, \sigma_0) = m(\sigma)$, donc $m(\sigma)$ ne dépend que de l'orbite $\Orb$ de $\sigma$, on notera $m(\Orb)$ cette quantité. De même, on a $X_0(\chi\sigma) = X_0(\sigma)$, donc on note $X_0(\Orb)$ ce groupe. Et pour tout $\chi \in X(M)$
\[
	\Ind_{M^0}^M \, \chi_0 \sigma_0 \simeq m(\Orb) \bigoplus_{\psi \in X_0(M)/X_0(\Orb)} \psi \chi \sigma
\]
On peut donc définir une application $\Lambda : \Theta_0 \to \Theta$ telle que pour tout $(P^0, \Orb_0) \in \Theta_0$, en notant $(P, \Orb) = \Lambda((P^0, \Orb_0))$, alors $P^0 \subset \cdot P$ (l'unicité de $P$ est assurée par \cite{Goldberg} proposition 2.10) et pour tout $\sigma_0 \in \Orb_0$, $\text{JH}(\Ind_{M^0}^M \sigma_0) \subset \cdot \Orb$. Par réciprocité de Frobenius, on voit que l'application $\Lambda$ est surjective et d'après le lemme \ref{lemme:decomposition_restriction}, on sait que l'image réciproque de $(P, \Orb) \in \Theta$ est une classe de $M$-conjugaison dans $\Theta_0$. Et comme $g \sigma_0 \hookrightarrow (g \sigma)_{|gM^0}$ pour tout $g \in G$, on vérifie que $\Lambda$ induit une bijection $\overline{\Lambda} : \Theta_0/G \to \Theta/G$.

Soit $(P^0, \Orb_0) \in \Theta_0$. Si $\psi : \Orb_0 \to \C$ est une fonction méromorphe dont tous les pôles sont contenus dans un compact de bord orienté $\Gamma$, on note
\[
	\int_{\Orb_0} \psi(\sigma_0) \, \D \sigma_0 = \int_{\Gamma} \psi(\sigma_0) \, \D \sigma_0
\]
Le théorème des résidus garantit que l'intégrale ne dépend pas du choix du compact contenant les pôles (ce qui justifie donc de l'omettre dans la notation). Si $\psi : \Orb \to \C$ est une fonction méromorphe et si $(P, \Orb) = \Lambda((P^0, \Orb_0))$, alors la fonction $\sigma \mapsto \sum_{\chi \in X_0(M)}  \psi(\chi \sigma)$ ne dépend que de l'orbite de $\sigma$ sous $X_0(M)$. On peut donc la voir comme une fonction sur $X(M^0)$, et noter
\[
	\int_{\Orb} \psi(\sigma) \ \D \sigma =  \int_{\Orb_0} \sum_{\chi \in X_0(M)}  \psi(\chi \sigma)\ \D \sigma_0
\]
% subsubsection notations (end)

\subsubsection{Une formule pour $f_s$} % (fold)
\label{ssub:une_formule_pour_f_s_}

Soient $s \in S$, $\sigma_0 \in \Orb_0$, en notant $\pi_{\sigma_0} = \Ind_{P^0}^{G^0} \, \sigma_0$ et $\Pi_{\sigma_0} = \Ind_{P^0}^{G} \, \sigma_0$, la fonction $\sigma_0 \mapsto \Tr(\Pi_{\sigma_0}(g^{-1} s) \, \zeta_s(\pi_{\sigma_0}))$ est rationnelle % Justification que les pôles sont dans un compact.
(ici on identifié $\zeta_s(\pi_{\sigma_0})$ à un élément de $\End_{\C}(\Pi_{\sigma_0})$ via le plongement évident $\pi_{\sigma_0} \hookrightarrow \Pi_{\sigma_0}$), et la fonction $g \mapsto \Tr(\Pi_{\sigma_0}(g^{-1} s) \, \zeta_s(\pi_{\sigma_0}))$ est à support dans $s G^0$ (et vaut $\Tr(\pi_{\sigma_0}(g^{-1} s) \, \zeta_s(\pi_{\sigma_0}))$ sur $sG^0$). On en déduit d'après \cite{Heiermann} théorème 3.2, que pour tout $g \in G$
\begin{equation}
	f_s(s^{-1} g) = \sum_{(P^0, \Orb_0) \in \Theta_0} c(\Orb_0) \int_{\Orb_0} \Tr(\Pi_{\sigma_0}(g^{-1} s) \, \zeta_s(\pi_{\sigma_0})) \, \D \sigma_0
\end{equation}
Où $c(\Orb_{\sigma_0}) > 0$ est une constante strictement positive donnée dans \cite{Heiermann} 3.2 par
\[
	c(\Orb_0) =  \frac{ |W^{M^0}| . |W(M^0, \Orb_0)| . \gamma(G^0 | M^0) . d(\Orb_0)}{|\mathcal{P}(M^0)| . |W^{G^0}| . |\text{Stab}_{X(M^0)}(\Orb_0)|}
\]
On vérifie que pour tout $g \in G$, $c(g \Orb_0) = c(\Orb_0)$. On notera donc $c([\Orb_0])$ cette quantité. On réunit les termes par $G$-orbites.
\[
	f_s(s^{-1} g) = \sum_{[P^0, \Orb_0] \in \Theta_0/G} \sum_{t \in G/G(t \cdot \Orb_0)} c([\Orb_0]) \int_{t \cdot \Orb_0} \Tr(\Pi_{\sigma_0}(g^{-1} s) \, \zeta_s(\pi_{\sigma_0})) \, \D \sigma_0
\]
Or pour $t \in G$, on a $\Pi_{\sigma_0} = \Ind_{P^0}^G \, \sigma_0 \underset{\lambda(t)}{\simeq} \Ind_{t \cdot P^0}^G \, t \cdot \sigma_0 = \Pi_{t \cdot \sigma_0}$, donc 
\begin{align*}
	\int_{t \cdot \Orb_0} \Tr(\Pi_{\sigma_0}(g^{-1} s) \, \zeta_s(\pi_{\sigma_0})) \, \D \sigma_0 &= \int_{\Orb_0} \Tr(\Pi_{t \cdot \sigma_0}(g^{-1} s) \, \zeta_s(\pi_{t \cdot \sigma_0})) \, \D \sigma_0 \\
	&= \int_{\Orb_0} \Tr(\Pi_{\sigma_0}(g^{-1} s) \, \zeta_s(t \cdot \pi_{\sigma_0})) \, \D \sigma_0
\end{align*}
Et d'après le lemme \ref{lemme:0.2s_Ind_P0}, on a
\begin{align*}
	\sum_{t \in S} \int_{t \cdot \Orb_{\sigma_0}} \Tr(\Pi_0(g^{-1} s) \, \zeta_s(\pi_{\sigma_0})) \, \D \sigma_0 &= \int_{\Orb_0} \Tr \pars{\Pi_{\sigma_0}(g^{-1} s) \, \bigoplus_{t \in S} \zeta_s(\pi_{t \cdot \sigma_0}) }\, \D \sigma_0 \\
	&= \int_{\Orb_0} \Tr(\Pi_{\sigma_0}(g^{-1} s) \, \zeta_s(\Pi_{\sigma_0})) \, \D \sigma_0
\end{align*}
Et donc, en posant $c'([\Orb_0]) = \frac{c([\Orb_0])}{|G(\Orb_0)/G^0|}$, on a
\begin{equation} \label{eq:formule_fs}
	f_s(s^{-1} g) = \sum_{[P^0, \Orb_0] \in \Theta_0/G} c'([\Orb_0]) \int_{\Orb_0} \Tr(\Pi_{\sigma_0}(g^{-1} s) \, \zeta_s(\Pi_{\sigma_0})) \, \D \sigma_0
\end{equation}

\subsubsection{Une formule pour $f$} % (fold)
\label{ssub:une_formule_pour_f_}

\begin{prop}
	\label{prop:formule_inversion}
	Il existe des constantes $C([\Orb]) > 0$ pour $[P, \Orb] \in \Theta/G$ telles que l'on a
	\begin{equation}
		f(g) = \sum_{[P, \Orb] \in \Theta/G} C([\Orb]) \int_{\Orb} \Tr \pars{\Pi_{\sigma}(g^{-1})  \zeta(\Pi_{\sigma})} \, \D \sigma
	\end{equation}
\end{prop}
\begin{proof}
	En sommant l'équation (\ref{eq:formule_fs}) sur $s \in S$, on tire
	\begin{align*}
		f(g) &= \sum_{s \in S} f_s(s^{-1} g) \\
		&= \sum_{s \in S}  \sum_{[P^0, \Orb_0] \in \Theta_0/G} c'([\Orb_0]) \int_{\Orb_0} \Tr(\Pi_{\sigma_0}(g^{-1} s) \, \zeta_s(\Pi_{\sigma_0})) \, \D \sigma_0 \\
		&=  \sum_{[P^0, \Orb_0] \in \Theta_0/G} c'([\Orb_0]) \int_{\Orb_0} \Tr \pars{\Pi_{\sigma_0}(g^{-1})  \sum_{s \in S }\Pi_{\sigma_0}(s) \, \zeta_s(\Pi_{\sigma_0})} \, \D \sigma_0
	\end{align*}
	Et en vertu du lemme \ref{lemme:0.2_Ind_P0}, on tire finalement
	\begin{equation}
		f(g) = \sum_{[P^0, \Orb_0] \in \Theta_0/G} c'([\Orb_0]) \int_{\Orb_0} \Tr \pars{\Pi_{\sigma_0}(g^{-1})  \zeta(\Pi_{\sigma_0})} \, \D \sigma_0
	\end{equation}
	On note $[P, \Orb] = \overline{\Lambda}([P^0, \Orb_0])$ comme en section \ref{ssub:notations}. Puisque $\overline{\Lambda}$ établit une bijection entre $\Theta_0/G$ et $\Theta/G$, on peut noter $c'([\Orb]) = c'([\Orb_0])$ sans équivoque.
	% 
	% Pour $\sigma_0 \in \Orb_0$, on choisit $\sigma \in \text{JH}(\Ind_{M^0}^M \sigma_0)$ de manière $X^0(M)$-équivariante (i.e. à $\chi_0\sigma_0$ on associe $\chi \sigma$) . Alors on peut écrire pour tout $\sigma_0 \in \Orb_0$, on peut écrire
	% \[
	% 	\Ind_{P^0}^{G^0} \, \sigma_0 = m(\sigma) \bigoplus_{\chi \in X_0(M)/X_{\sigma}(M)} \chi \sigma
	% \]
	% où $X_{\sigma}(M) = \braces{\chi \in X_0(M), \chi \sigma \simeq \sigma}$. Comme $m(\sigma) = \dim \Hom_{G^0}(\sigma_0, \sigma_{|G^0})$ et que l'espace $\Hom_{G^0}(\chi_0\sigma_0, (\chi\sigma)_{|G^0})$ ne dépend pas de $\chi \in X(M)$, alors $m(\sigma)$ ne dépend en fait que de $[\Orb]$, on le notera donc $m([\Orb])$. De même, le groupe $X_{\chi\sigma}(M)$ ne dépend pas de $\chi \in X(M)$, on le notera donc $X(\Orb$.
	En notant $\Pi_{\sigma} = \Ind_P^G \, \sigma$ pour $\sigma \in \Orb$, on a alors
	\begin{align*}
		f(g) &= \sum_{[P^0, \Orb_0] \in \Theta_0/G} c'([\Orb]) \int_{\Orb_0} m([\Orb]) \sum_{\chi \in X_0(M)/X_0(\Orb)} \Tr \pars{ \Pi_{\chi\sigma}(g^{-1})  \zeta(\Pi_{\chi \sigma})} \, \D \sigma_0 \\
		&= \sum_{[P^0, \Orb_0] \in \Theta_0/G} \frac{c'([\Orb]) \, m([\Orb])}{|X_0(\Orb)|} \int_{\Orb_0} \sum_{\chi \in X_0(M)} \Tr \pars{ \Pi_{\chi\sigma}(g^{-1})  \zeta(\Pi_{\chi \sigma})} \, \D \sigma_0 
	\end{align*}
	On remarque que pour $t \in S$, on a $X_0(t \cdot \Orb) = t \cdot X_0(\Orb)$, donc le cardinal $|X_0(\Orb)|$ ne dépend en fait que $[\Orb]$. En notant ``$\int_{\Orb} \ \D \sigma = \int_{\Orb_0} \sum_{\chi \in X_0(M)} \ \D \sigma_0$'', et $C([\Orb]) = \frac{c'([\Orb]) \, m([\Orb])}{|X_0(\Orb)|} > 0$, on peut alors réécrire
	\[
		f(g) = \sum_{[P, \Orb] \in \Theta/G} C([\Orb]) \int_{\Orb} \Tr \pars{\Pi_{\sigma}(g^{-1})  \zeta(\Pi_{\sigma})} \, \D \sigma
	\]
	Puis en tirant parti de la bijection $\overline{\Lambda}$ définie en \ref{ssub:notations},
	\begin{equation*}
		f(g) = \sum_{[P, \Orb] \in \Theta/G} C([\Orb]) \int_{\Orb} \Tr \pars{\Pi_{\sigma}(g^{-1})  \zeta(\Pi_{\sigma})} \, \D \sigma
	\end{equation*}
\end{proof}

\bibliographystyle{plain}
\bibliography{biblio_clean}

\end{document}